\documentclass[11pt]{article}    
\usepackage{amsmath, amssymb, amsthm,pdfsync}
\usepackage{fullpage}
\usepackage{verbatim}
\usepackage{graphicx}
\usepackage{subfig, caption}
\usepackage{cancel}
\usepackage{enumitem}
\usepackage{wrapfig}
\usepackage{mathtools}
\usepackage{xcolor}
\usepackage{epigraph}
\usepackage{cleveref}
\usepackage{overpic}
\usepackage{esint}
\usepackage{dsfont}
\usepackage{color}

\setlist[itemize]{itemsep=-1mm}

\newtheorem{theorem}{Theorem}
\newtheorem{lemma}[theorem]{Lemma}
\newtheorem{lem}[theorem]{Lemma}
\newtheorem{prop}[theorem]{Proposition}

\newtheorem{proposition}[theorem]{Proposition}
\newtheorem{definition}[theorem]{Definition}
\newtheorem{remark}[theorem]{Remark}

\newtheorem{hypothesis}[theorem]{Hypothesis}

\usepackage{dsfont}    \newcommand{\1}{\mathds{1}}

\usepackage{enumitem}
\setlist{nolistsep}

\newcommand{\Id}{\mathrm{Id}}

\newcommand{\R}{\mathbb{R}}

\newcommand{\N}{\mathbb{N}}

\newcommand{\e}{\varepsilon}
\renewcommand{\epsilon}{\varepsilon}
\newcommand{\eps}{\varepsilon}

\newcommand {\Chi} {{\bf \raise 2pt \hbox{$\chi$}} }

\DeclareMathOperator{\argmin}{argmin}
\DeclareMathOperator{\Real}{Re}

\DeclareMathOperator{\Vol}{Vol}
\newcommand{\supp}{\mathrm{supp}\,}

\Crefname{assumption}{Assumption}{Assumptions}
\Crefname{theorem}{Theorem}{Theorems}
\Crefname{lemma}{Lemma}{Lemmas}
\Crefname{corollary}{Corollary}{Corollaries}
\Crefname{proposition}{Proposition}{Propositions}
\Crefname{theorem}{Theorem}{Theorems}
\Crefname{conjecture}{Conjecture}{Conjectures}
\Crefname{hypothesis}{Hypothesis}{Hypotheses}

\title{Influence of a mortality trade-off on the spreading rate of cane toads fronts}
\date\today                                           
\author{Emeric Bouin\,
\footnote{CEREMADE - Universit\'e Paris-Dauphine, UMR CNRS 7534, PSL Research University, Place du Mar\'echal de Lattre de Tassigny, 75775 Paris Cedex 16, France. E-mail: \texttt{bouin@ceremade.dauphine.fr}}\; 
\and
Matthew H. Chan\footnote{School of Mathematics and Statistics, 
University of Sydney, NSW 2006, Australia. E-mail:
\texttt{M.Chan@maths.usyd.edu.au}}\;
\and
Christopher Henderson\footnote{Corresponding author} \footnote{Department of Mathematics, the University of Chicago, Chicago, IL 60637, E-mail: \texttt{henderson@math.uchicago.edu}}\;
\and
Peter S. Kim \footnote{School of Mathematics and Statistics, University of Sydney, Australia. E-mail:
\texttt{peter.kim@sydney.edu.au}}\;
}

\begin{document}
\maketitle

\begin{abstract}
\noindent We study the influence of a mortality trade-off in a nonlocal reaction-diffusion-mutation equation that we introduce to model the invasion of cane toads in Australia. This model is built off of one that has attracted attention recently, in which the population of toads is structured by a phenotypical trait that governs the spatial diffusion. We are concerned with the case when the diffusivity can take unbounded values and the mortality trade-off depends only on the trait variable. Depending on the rate of increase of the penalization term, we obtain the rate of spreading of the population. We identify two regimes, an acceleration regime when the penalization is weak and a linear spreading regime when the penalization is strong. While the development of the model comes from biological principles, the bulk of the article is dedicated to the mathematical analysis of the model, which is very technical.  The upper and lower bounds are proved via the Li-Yau estimates of the fundamental solution of the heat equation with potential on Riemannian manifolds and a moving ball technique, respectively, and the travelling waves by a Leray-Schauder fixed point argument.  We also present a simple method for {\em a priori} $L^\infty$ bounds.
\end{abstract}
\noindent{\bf AMS Class. No:} {35Q92, 45K05, 35C07}

%
\section{Introduction}
\subsection{Model and biological background}

The invasion of cane toads in Australia has singular features that are much different from standard spreading observed in most other
species. Data from field biologists \cite{phillips2006invasion, Shine} show that the invasion speed has steadily increased 
during the eighty years since the toads were introduced in Australia.  Moreover, they found that toads at the edge of the front have much longer legs.  This is just one example of a non-uniform space-trait 
distribution (see also the expansion of bush crickets in Britain~\cite{Thomas}). The current biological literature now states that this is due to "spatial sorting": the offspring produced at the edge of the front appear to have higher mean dispersal rate. 

It has been demonstrated by biologists that increased dispersal is often associated with reduced investment in reproduction, for example in populations of the peckled wood butterfly, \textit{Pararge aegeria} \cite{Hughes}. Some physiological description of this trade-off between dispersal and fecundity has been reported in \cite{Mole}. There, two morphs of the cricket \textit{Gryllus rubens} (Orthoptera, Gryllidae) are studied: a fully-winged (flight-capable morph) and a short-winged morph (that cannot fly).  It turns out that the short-winged morph is substantially more fecund than the fully-winged one. This widespread occurrence of dispersal polymorphisms among insects is consistent with the fact that fitness costs are associated with flight capability. It is now well documented that for both males and females (for example for the planthopper \textit{Prokelisia dolus}) there is a strong trade-off between flight capability and reproduction \cite{Langellotto}. See also \cite{Guerra} where the physiological differences between the male and female of two morphs of crickets that may lead to such a trade-off are discussed. We refer to \cite{Bonte} for an extensive review on the different cost types that occur during dispersal in a wide array of organisms, ranging from micro-organisms to plants and invertebrates to vertebrates.


In view of these biological issues, we are interested in the influence of a mortality trade-off on the rate of spreading of a structured population.  Namely, our goal is to estimate of the effect of this penalization; that is, depending on the strength of the trade-off, does the population go extinct or still propagate, and in the latter case, what is the effect on the acceleration seen in~\cite{BerestyckiMouhotRaoul,BHR_Acceleration}?

To answer these questions, we focus on a cane toads equation with mortality trade-off. This is a nonlocal reaction-diffusion-mutation equation that is a refinement of the now standard cane-toad equation proposed in~\cite{BenichouEtAl} and investigated in~\cite{BerestyckiMouhotRaoul,BouinCalvez,BHR_LogDelay,BHR_Acceleration,Turanova} (see also \cite{ArnoldDesvillettesPrevost, BouinMirrahimi, ChampagnatMeleard, LamLou, 
PerthameSouganidis} for similar studies).  We now introduce this model. 
The population density $n$ is structured by a spatial variable, $x\in \R$, and a 
motility variable $\theta\in \overline\Theta$, where $\Theta\stackrel{\rm def}{=}(\underline\theta, \infty)$, with a  fixed   $\underline\theta>0$. The spatial diffusivity is  exactly $\theta$, representing the effect of the variable motility on the spreading rates of the species.
This population may reproduce, with a free growth rate $r$, in such a way that a parent gives its trait to its offspring up to phenotypical mutations (that is, mutations on the motility variable), that are modelled here with a diffusion in the trait variable $\theta$, with a variance $\sigma^2$. In addition, each toad competes locally in space with all other individuals for resources. This introduces the nonlocality in the model. Finally, and this is the specificity of this paper, we take into account a mortality trade-off, denoted by $m$, that penalizes high traits for reproduction. We are thus led to considering the following problem:
\begin{equation*}
\begin{cases}
n_t = \theta n_{xx} + \alpha n_{\theta\theta} + r n \left( 1 - m(\theta) - \rho\right), &\qquad (t,x,\theta) \in \R^+ \times \R \times \Theta,\\
n_\theta(t,x,\underline\theta) = 0, & \qquad (t,x) \in \R^+ \times \R. 
\end{cases}
\end{equation*}
where the total population at time $t$ and position $x$ is
\begin{equation*}\label{eq:rho}
\rho(t,x)=\int_{\underline\theta}^\infty n(t,x,\theta)d\theta.
\end{equation*} 
The trait diffusivity is $\alpha = r \sigma^2$. The equation is complemented with Neumann boundary conditions at $\theta = \underline\theta$, since lower traits should not be reachable. After a suitable rescaling, one can reduce the problem to studying the equation with $\alpha = 1$ and $r=1$.  The resulting equations are
\begin{equation}\label{eq:main}
\begin{cases}
n_t = \theta n_{xx} + n_{\theta\theta} +  n \left( 1 - m(\theta) - \rho\right), &\qquad (t,x,\theta) \in \R^+ \times \R \times \Theta,\\
n_\theta(t,x,\underline\theta) = 0, & \qquad (t,x) \in \R^+ \times \R. 
\end{cases}
\end{equation}
We now describe the class of trade-off terms $m$ that we consider in this paper.
\begin{hypothesis}\label{hyp:m} We assume that $m$ depends only on $\theta$, that $m(\underline\theta) = 0$ and that $m \in \mathcal{C}^2(\overline \Theta)$ increases to $+\infty$ as $\theta$ tends to $\infty$. Moreover, we suppose that $\lim_{\theta\to\infty} m(\theta)/\theta$ exists and is an element of $\R^+\cup\{+\infty\}$ and that if $m(\theta)/\theta$ tends to zero as $\theta$ tends to $+\infty$, then $m'' \in L^\infty\left(\overline\Theta\right)$ and there exists $\theta_d > \underline\theta$ such that $m(\theta)/\theta$ is decreasing for all $\theta \geq \theta_d$.
\end{hypothesis}
Importantly, we point out that the fact that $m$ tends to infinity gives a positive growth rate to only small values of $\theta$. For the entirety of this work, we assume that $m$ always satisfies~\Cref{hyp:m}.  An important class of examples of $m$ satisfying~\Cref{hyp:m} are $m(\theta) = C(\theta^p - \underline\theta^p)$ for $C, p > 0$.

Our main question for the rest of this work is, if $n(0,\cdot) \equiv n_0$ is a nonzero, nonnegative initial condition such that there exists $C_0 > 0$ such that 
\begin{equation}\label{eq:n_0}
		n_0 \leq C_0 \1_{[\underline\theta, \underline\theta+C_0]\times[-\infty,C_0]},
\end{equation}
then at what speed does the population $n$ propagate?

A related model for a host-parasite system with mortality trade-off has been discussed numerically and formally by Chan \textit{et al.}~in \cite{Chan}. The results there confirm  observations made in empirical and agent-based studies that spatial sorting can still occur with a disadvantage in reproductivity and/or survival in more motile individuals. Moreover, we find that such a disadvantage in reproductivity and/or survival is unlikely to be large if spatial sorting is to have a noticeable effect on the rate of range expansion, as it has been observed to have over the last 60 years in northern Australia. The results of the present paper prove these findings and enlighten them quantitatively.

\subsection{Heuristics and main results}

\subsubsection*{A condition for non extinction} 

In order to expect propagation at any speed, the average growth rate $\gamma_\infty$ in time should be positive. This is necessary to expect any positive steady state at the back of a traveling front. In other words, letting $Q$ and $\gamma_\infty$ be the principal eigenfunction and eigenvalue of the linearized equation
\begin{equation}\label{eq:specQ}
\begin{cases}
Q'' +  (1 - m) \, Q  = \gamma_\infty Q, \qquad \text{on } \Theta,\\
Q'\left(\underline \theta \right) = 0,~~~ Q > 0,
\end{cases}
\end{equation}
we expect propagation {\em only} in the case that $\gamma_\infty>0$.  Notice that, since $m \geq 0$ and $m\not\equiv 0$, $\gamma_\infty < 1$.  The sign of $\gamma_\infty$ does not depend strongly on the growth of $m$ at $\infty$, but on having sufficiently many traits with sufficiently large growth rates\footnote{Indeed, consider the following simple example.  When $m(\theta) = \sigma^2 (\theta-\underline\theta)^2$ for any $\sigma>0$, one can find the principal eigenfunction and eigenvalue explicitly: $Q(\theta) = \exp\{- \sigma(\theta-\underline\theta)^2/4\}$ and $\gamma_\infty = 1-\sigma$.  Though, for any value of $\sigma>0$, $m$ has quadratic growth, $\gamma_\infty$ can be positive or negative depending on $\sigma$.}.  Our expectation above is confirmed by the following:

\begin{prop}\label{prop:extinction}
Suppose that $\gamma_\infty \leq 0$ and $\supp(n_0)$ is compact.  Then
\begin{equation*}
\lim_{t \to \infty} \sup_{(x,\theta)\in\R\times \Theta} n(t,x,\theta) = 0.
\end{equation*}
\end{prop}

We also note that, when $Q$ is normalized with $\int_{\Theta} Q(\theta) \, d\theta = \gamma_\infty$ and $\gamma_\infty > 0$, this eigenvector is expected to be the limit of the population density behind the front.

\subsubsection*{The linear propagation regime}

In the case of non-extinction, i.e.~when $\gamma_\infty > 0$, one may expect propagation of the initial population.  A first attempt in this direction is to look for travelling waves. Since the problem is of Fisher-KPP type, we may expect at first glance that any travelling front is a pulled front. As a consequence, in order to compute the possible speed of propagation of such a front, we follow the classical strategy by linearizing \eqref{eq:main} around $0$ and looking for solutions of the form $e^{-\lambda(x-c_\lambda t)} Q_\lambda(\theta)$.  For any given $\lambda > 0$,  we apply the Krein-Rutman theorem to solve the spectral problem
\begin{equation}\label{eq:specQlambda}
\begin{cases}
		Q_\lambda'' + \left[  \lambda^2 \theta - \lambda c_\lambda + (1 -  m(\theta)) \right] Q_\lambda  = 0, &\qquad \theta \in \Theta, \\
Q_\lambda' \left( \underline \theta \right) = 0, \qquad Q_\lambda > 0&
\end{cases}
\end{equation}
when $m$ increases sufficiently quickly. 
%
%
%
This leads to the following theorem. 
\begin{theorem}\label{prop:tw}
Suppose that $m$ satisfies \Cref{hyp:m}, that $\gamma_\infty > 0$ and that $\lim_{\theta \to \infty} m(\theta)/\theta$ is positive. Then \eqref{eq:main} admits a travelling wave solution $(c^*,\mu)$, with $c^*$ defined in \Cref{sec:tw}.  In other words $n(t,x,\theta) = \mu(x-c^*t,\theta)$ solves~\eqref{eq:main}, with $c^* >0$, and
	\begin{equation}\label{eq:limits_of_wave}
		\liminf_{\xi\to-\infty} \mu(\xi,\underline\theta)
			> 0
			\qquad\text{ and }\qquad
			\lim_{\xi\to\infty} \sup_{\theta\in\Theta} \mu(\xi,\theta) = 0.
	\end{equation}
\end{theorem}

As with the standard Fisher-KPP equation, we expect this speed $c^*$ to be the minimal speed in the sense that if $c \geq c^*$, then there is a travelling wave of speed $c$.  Since this is not our main interest, we do not address it here.

Our main interest is in a spreading result for the Cauchy problem \eqref{eq:main}. We thus ask whether the travelling wave constructed in \Cref{prop:tw} is stable.  This is answered by the following theorem.
\begin{theorem}\label{thm:cauchy_finite}
	Suppose the conditions of \Cref{prop:tw} hold. Suppose that $n$ solves~\eqref{eq:main} with initial conditions satisfying~\eqref{eq:n_0}.  Then there exists $\underline n>0$ such that for every $\epsilon >0$, we have
	\[
		\liminf_{t\to\infty} \inf_{|x|\leq (c^*-\epsilon)t} n(t,x,\underline\theta) \geq \underline n,
			\qquad\text{ and }\qquad
			\lim_{t\to\infty}\sup_{x\geq (c^* + \epsilon)t} \sup_{\theta \in\Theta} n(t,x,\theta) = 0.
	\]
\end{theorem}
This type of result is standard going back to~\cite{AronsonWeinberger2} in the local Fisher-KPP setting. Since the dynamics of the solution are so complicated, it would be interesting to obtain more precise estimates on the propagation speed.  We expect a logarithmic delay {\em a la} Bramson, see~\cite{BHR_LogDelay} for the delay in the cane toads equation and references therein for more general settings.

We briefly outline the proofs of \Cref{prop:tw} and \Cref{thm:cauchy_finite} and the main difficulties that we have to face.  The construction of a travelling wave solution with minimal speed of \Cref{prop:tw} is done by building a solution to an approximate problem on a finite ``slab'' by a degree theory fixed point argument. This construction appears to be a non-trivial extension of the one for the cane toads equation with bounded traits \cite{BouinCalvez}. In particular, our proof differs from the usual procedure in \cite{BouinCalvez} because we have both unbounded diffusivity and unbounded growth rates. In this direction, we also point out connections to \cite{AlfaroBerestyckiRaoul,AlfaroCovilleRaoul,BerestyckiJinSilvestre},  where travelling waves for structured models were constructed.  The proof of the spreading result in \Cref{thm:cauchy_finite} proceeds as follows. We directly construct a super-solution of $n$ using~\eqref{eq:specQlambda}, which provides the upper bound. 
The lower bound follows by building a solution to a related problem on a moving ball to using the intermediate steps of the construction of the travelling wave and applying a local-in-time Harnack inequality to compare this to $n$. This is a simpler version of the kind of procedure described in greater length in the next section.

The proof of \Cref{prop:tw} is in \Cref{sec:tw} and the proof of \Cref{thm:cauchy_finite} is in \Cref{sec:cauchy_finite}.

\subsubsection*{The acceleration regime}

The condition for the existence of travelling waves may be roughly re-written as $m$ is at least linear. On the other hand, when $m$ is sub-linear and $\gamma_\infty>0$, we still expect propagation. Since the spectral problem \eqref{eq:specQlambda} is not solvable, we may expect an acceleration phenomenon exactly as for the cane toads equation \cite{BCMetal,BHR_Acceleration}.  Before stating our main result, we give some heuristics that make it appear naturally.  Returning to the linearized equation  
\begin{equation}\label{eq:linearized}
	\overline n_t = \theta \overline n_{xx} + \overline n_{\theta\theta} + \overline n(1 - m(\theta))
\end{equation}
and recalling the definition of the pair $(\gamma_\infty,Q)$, we see that a function of the form $\overline{n} \propto Q(\theta) e^{\gamma_\infty t}$ solves \eqref{eq:linearized}. This gives a first bound of spreading in the direction $\theta$. Indeed, by writing $\psi = -\log(Q)$, one may check that $\psi$ satisfies $-\psi'' + |\psi'|^2 + (1-m) = \gamma_\infty$ and hence%
\footnote{To see the lower bound: let $R = \psi'$. Then $R$ satisfies $R' = R^2 + \left( 1 - \gamma_\infty - m \right)$ for $\theta > \underline{\theta}$ and $R(\underline{\theta}) = 0$. For $\theta \leq m^{-1}(1-\gamma_\infty)$, $R$ is increasing, and thus positive. For $\theta > m^{-1}(1-\gamma_\infty)$, $R$ satisfies $R' = R^2 - R_0^2$, with $R_0 = \left(m - \left(1 - \gamma_\infty\right)\right)^\frac12$. One can see that $R$ can not cross the curve $R_0$. Indeed, if it did, $R$ would become decreasing with $R'$ tending to $-\infty$, and thus negative for sufficiently large $\theta$, which is impossible since $Q$ is integrable. We thus get that $\psi(\theta) \geq \int_{\underline{\theta}}^\theta \left[\max\left( 0 , m(s) - \left(1 - \gamma_\infty\right) \right) \right]^\frac12 \, ds$. %
To see the upper bound: fix $\epsilon>0$, write $R = \mu\sqrt{m}$, and observe $\mu' \sqrt m = (\mu^2 -1)m + 1- \gamma_\infty - \mu m'/(2\sqrt m)\geq (\mu^2-1)m + 1-\gamma_\infty - \epsilon \mu^2 - C_\epsilon(m'/\sqrt m)^2$.  Choose $\theta_\epsilon$ such that, if $\theta \geq \theta_\e$, then $C_\epsilon (m'/\sqrt m)^2 < 1-\gamma_\infty$ and $m > 2(1+\epsilon)$.  Then if, for any $\theta \geq \theta_\epsilon$, $\mu(\theta) \geq \sqrt{1+\epsilon}$, after some computation we see that $\mu' \geq \epsilon \mu^2 \sqrt{m} / (2(1+\epsilon))$ and $\mu$ blows up at a some $\theta_b> \theta_\epsilon$.  This implies that $Q$ has an interior zero at $\theta_b$.  This is a contradiction, implying that $\mu \leq \sqrt{1+\epsilon}$ for all $\theta \geq \theta_\e$.  The estimate follows.},
as $\theta\to\infty$,

\begin{equation}\label{eq:Q_asymptotics}
	\psi(\theta) = \int_{\underline{\theta}}^\theta \sqrt{m(s)}~ds + o(\theta \sqrt{m(\theta)}).
\end{equation}
%
It is thus natural to define what corresponds to a spreading rate in the direction $\theta$ as follows.

\begin{definition}
For $\theta \in \overline\Theta$, define $\Phi(\theta) = \int_{\underline{\theta}}^\theta \sqrt{m(s)}ds$.  Fix any time $t>0$, any constant $a>0$, and define $\eta_a(t) \in \overline{\Theta}$ to be the unique solution of $\Phi(\eta_a(t)) = at$. When $a=1$, we denote $\eta(t) = \eta_1(t)$.
\end{definition}

Returning to $\overline n$, above, we see that, heuristically, the spreading in $\theta$ should be at least $O(\eta(t))$ since this is where the decay in $Q$ balances the growth of $e^{\gamma_\infty t}$.

Importantly, when $m(\theta)/\theta$ tends to $0$, there exists $D_m$ such that if $\theta$ is sufficiently large then
\begin{equation}\label{eq:phi_to_m}
	\Phi(\theta)
		\leq  \theta \sqrt{m(\theta)}
		\leq D_m \Phi(\theta).
\end{equation}		
Due to~\eqref{eq:phi_to_m}, we find a constant $C_a$ such that when $t$ is sufficiently large then
\begin{equation}\label{eq:etaa_to_eta1}
		C_a^{-1} \eta(t)
			\leq \eta_a(t)
			\leq C_a \eta(t).
\end{equation}
We note that, in the proof of \Cref{thm:acceleration}, it is more convenient to deal with the family $\eta_a(t)$ than to look at the scalar multiples of $\eta(t)$.  Due to~\eqref{eq:etaa_to_eta1}, these approaches are equivalent.

Now, heuristically, knowing the natural scaling between space and trait variables for the standard cane toads equation, we expect a propagation in space to be $O(\eta(t)^\frac{3}{2})$ since we expect propagation in trait to be $O(\eta(t))$. Our next result, and the main focus of this work, confirms these heuristics.

\begin{theorem}\label{thm:acceleration}
	Suppose that $m$ satisfies \Cref{hyp:m}, $\gamma_\infty > 0$,  $m(\theta)/\theta$ tends to zero as $\theta$ tends to $+\infty$, and $n$ satisfies~\eqref{eq:main}-\eqref{eq:n_0}.  Then there exist positive constants $\underline n$, $\underline c$, and $\overline c$ such that
	\[
		\liminf_{t \to \infty} \inf_{|x| \leq \underline c\eta(t)^{3/2}} n(t,x,\underline\theta) \geq \underline n
			\qquad\text{ and }\qquad
			\lim_{t\to \infty} \sup_{x\geq \overline{c} \eta(t)^{3/2}} \sup_{\theta \in \Theta} \ n(t,x,\theta) = 0.
	\]
\end{theorem}

This result might be surprising at first glance, since populations with very high traits have a negative growth rate. It turns out that the spatial sorting still gives a strong propagation force to population with high traits at the edge of the invasion. 
%
%


To illustrate the result, we discuss two concrete choices of $m$.  First, if $m(\theta) \sim \theta^p$ for $p\in(0,1)$, one can check that $\eta(t) \sim t^{2/(2+p)}$.  Hence the front is at $\eta^{3/2}(t) \sim t^{3/(2+p)}$.  We point out that this is an  interpolation between the cases $p=1$, when no acceleration occurs (see~\Cref{thm:cauchy_finite}), and $p=0$, when acceleration is of order $t^{3/2}$ (see~\cite{BerestyckiMouhotRaoul,BHR_Acceleration})\footnote{Strictly speaking~\cite{BerestyckiMouhotRaoul,BHR_Acceleration} does not deal with the general case when $m(\theta)$ converges to a constant, but instead with the case $m(\theta) = 0$.  This particular case corresponds to the growth rate at infinity with $p = 0$.}.  Second, if $m(\theta) \sim \log(\theta)^p$ then it is easy to check that $\eta(t) \sim t\log(t)^{-p/2}$.  Hence the front is at $\eta(t)^{3/2} = t^{3/2}\log(t)^{-3p/4}$.  Notice that with this weaker trade-off term, the acceleration is almost at the same order as with no trade-off.

Let us now comment on the difficulties of the proof of \Cref{thm:acceleration}. In order to obtain the lower bound, we follow a similar strategy as in~\cite{BouinHenderson,BHR_Acceleration}.  We build a sub-solution on a moving ball using the principal Dirichlet eigenvalue.  There are three main difficulties here.  First, the problem is nonlocal and thus does not have a comparison principle.  To overcome this, we relate it to a local problem by estimating the nonlocal term $\rho$ using two ingredients: when $\theta$ is small, we may use a local-in-time Harnack inequality and when $\theta$ is large, we may obtain a priori estimates on the tails in trait of the solution $n$.  Second, in contrast to~\cite{BHR_Acceleration}, the path of this moving ``bump'' sub-solution cannot be found explicitly since the ODE system for the optimal path given by the Euler-Lagrange equation is not explicitly solvable.  Instead, we must optimize over rectangular paths.  Thirdly, the trade-off term $m$ is large when $\theta$ is large.  Hence, when $\theta$ is large, we add a multiplicative factor to our super-solution, which exponentially decays in time at a large rate, in order to absorb the high mortality rate.  We make up for this at the end of the trajectory by letting our moving sub-solution remain unmoving in a favorable area for a long period of time.  This strategy, absorbing losses and then re-growing later while keeping a careful accounting of them, is new in the reaction-diffusion literature to our knowledge.  In contrast, classical results come from settings where an eigenvalue can be well-posed and $O(1)$ solutions can be built using it.

The strategy of the proof of the upper bound is related in spirit to~\cite{BHR_Acceleration} but is technically completely different. 
In order to avoid complications with the nonlocal term, we notice that solutions to the linearized equation \eqref{eq:linearized}
are super-solutions to $n$.  As such we seek bounds on $\overline n$.  Historically, there are two ways to obtain super-solutions to reaction-diffusion equations: with a travelling wave solution or with heat kernel estimates.  Since a travelling wave moves at a constant speed, it cannot be used to bound an accelerating front from above.  On the other hand, classical heat kernel estimates on $\R^2$ require the diffusion operator to be comparable to the Laplacian and require any zeroth order terms in the operator to be bounded.  This is not our setting as $\theta \partial_x^2 + \partial_\theta^2$ is not comparable to the Laplacian uniformly in $\theta$ and $m(\theta)$ is unbounded by assumption.  In any case, our goal is understanding the precise balance between these two terms, so, even if it were possible to bound these by constants, this would provide too rough of an estimate for our purposes.  As such, we proceed by considering $\R\times\Theta$ as a two dimensional Riemannian manifold with boundary with the appropriate metric $g$.  After removing an integrating factor, we may view the linearized operator~\eqref{eq:linearized} as the Laplace-Beltrami operator $\Delta_g$ with a potential $-m(\theta)$.  We may then appeal to the methods of~\cite{LiYau} to obtain bounds on the fundamental solution of~\eqref{eq:linearized}.  After some careful modification of this fundamental solution and after reinserting the integrating factor, this provides a super-solution to $n$. We note that these heat kernel estimates do not provide the propagation result immediately.  Indeed, the results in~\cite{LiYau} give heat kernel estimates in terms of a Lagrangian, and this Lagrangian is itself difficult to estimate precisely. 
The use of the estimates coming from~\cite{LiYau} is not common in works investigating propagation, and we believe that this is an important addition to the toolbox for these types of problems.

Finally, we should mention why two other more well-known methods do not work.  First, one might ask if we can construct an explicit super-solution.  For example, one might expect functions of the same form as the upper estimate of the heat kernel in~\cite{LiYau} to be super-solutions. We outline in~\Cref{sec:acceleration} why this is not necessarily the case. 
Second, one may ask why large deviations methods, e.g.~the probabilistic methods of Friedlin~\cite{Freidlin1, Freidlin2} or the thin front limit approach of Evans and Souganidis~\cite{EvansSouganidis}, do not work.  In the construction of the sub-solution, we see that there are two scales.  Indeed, in the case where $m(\theta) \sim \theta^p$, the population that {\em drives} the acceleration is at $t^{3/(2-p)}$ while the bulk of the population is at $t^{3/(2+p)}$.  Hence scaling $(x,\theta)$ in the appropriate way in order to see the front loses the information about the population which drives the acceleration as it is lost in the scaling.  We should point out that this is an interesting feature of the model: the population that drives the front is different from the population that is sustained behind the front.

The proof of \Cref{thm:acceleration} is in \Cref{sec:acceleration}.

%
%

\subsubsection*{A Harnack estimate and a uniform upper bound}

Two key tools in our analysis are a uniform-in-time upper bound of $n$ and a local-in-time Harnack inequality that we can deduce from it.  We state both now.

\begin{prop}\label{prop:uniform_upper_bound}
	Suppose that $m$ satisfies~\Cref{hyp:m}.  Suppose that $n$ satisfies~\eqref{eq:main} with initial condition satisfying~\eqref{eq:n_0}.  Then there exists a constant $M$, depending only on $m$, such that
	\[
		\|\rho\|_{L^\infty(\R^+\times \R)}, \|n \|_{L^\infty(\R^+\times \R \times \Theta)} \leq M.
	\]
\end{prop}

\begin{lemma}\label{lem:harnack}
Let $\epsilon > 0$, $t_0 > 0$, $R>0$ and any point $x_0 \in \R$.  There exists $C_{R,\epsilon,t_0}$, depending only on $\epsilon$, $t_0$, and $R$, such that a solution $n$ of~\eqref{eq:main}-\eqref{eq:n_0} satisfies
 \[
	\rho(t,x_0)
		\leq \epsilon + C_{R,\epsilon,t_0} \inf_{|x-x_0| <R , \theta-\underline\theta < R} n(t,x,\theta),
		\qquad \text{ for all } t \geq t_0.
 \]
\end{lemma}

In general, it is difficult to obtain a uniform upper bound because there is no maximum principle of~\eqref{eq:main} due to the nonlocal term.  The bound must then be obtained by a careful understanding of the regularity of $n$ given a particular bound on $\rho$.  Specifically, one must use parabolic regularity estimates to show that if $n$ is large then $\rho$ is greater than $1$ and no maximum may occur. 

Our proof of Proposition~\ref{prop:uniform_upper_bound} is in the same spirit as the proofs in~\cite{BerestyckiMouhotRaoul, HamelRyzhik, Turanova} in that it relies on the natural scaling of the parabolic equation versus that of $\rho$, yet significantly simpler in presentation and technical details.  Indeed, by appealing to standard local regularity estimates in Sobolev spaces and the Gagliardo-Nirenberg interpolation inequality, we are able to avoid the involved technical details of~\cite{BerestyckiMouhotRaoul, Turanova} and obtain a succinct proof.

Lastly, we point out that the weak Harnack inequality referenced above allows us to compare $\rho$ and $n$ for bounded $\theta$, and thus, we can compare each solution of~\eqref{eq:main} to the solution of a related local problem.  In order to compare $\rho$ and $n$, there are two local-in-time Harnack inequalities which are useful for this,~\cite[Theorem~1.2]{BHR_LogDelay} and~\cite[Theorem~2.6]{AlfaroBerestyckiRaoul}.  We use the latter because, though less precise, it is sufficient for our purposes and it is significantly easier to prove.

Proposition~\ref{prop:uniform_upper_bound} and \Cref{lem:harnack} are used throughout this article.  Their proofs, which are independent of all other results, may be found in \Cref{sec:apriori}. 

\subsection*{Numerical simulations}

We end this introduction by showing numerical simulations that illustrate the different propagation regimes. For this, we use our typical example where $m(\theta) \sim \theta^p$ when $\theta$ tends to $+\infty$.  We present, for four different values of $p$ ($p=1/3, 2/3, 1, 4/3$) some plots of the solutions "from above" in the phase space $\R\times\Theta$ for various values of time. In case of a linear propagation, it is clear from these plots. In the acceleration regime, we also provide a figure with $\rho$ for various values of time, which helps viewing the accelerated propagation.

\begin{figure}[htbp]
\begin{center}
\includegraphics[width = .45\linewidth]{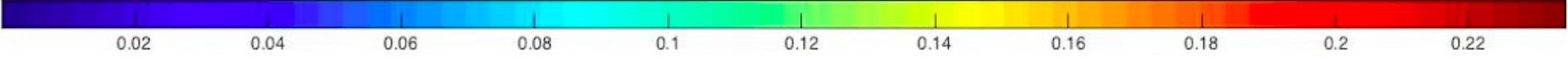}\qquad~~
\includegraphics[width = .45\linewidth]{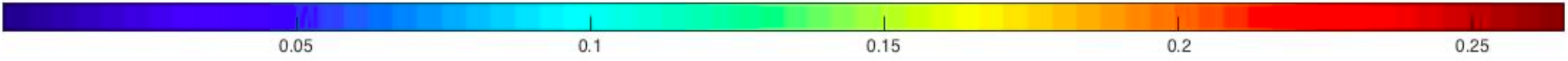}
~~\includegraphics[width=.99\linewidth]{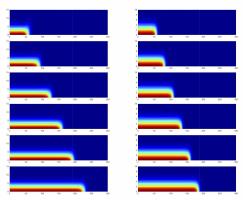}\\
\caption{Numerical simulations of the Cauchy problem of equation \eqref{eq:main} at a fixed time, in the phase space $\R\times\Theta$ at times $t=10$, $t=20$, $t=30$, $t=40$, $t=50$, $t=60$ with the choice $\underline \theta = .01$. Left column: $p=1$. Right column: $p=4/3$. Both exhibit propagation at a linear rate.}
\label{fig:Shape}
\end{center}
\end{figure}

%
\begin{figure}[htbp]
\begin{center}
\includegraphics[width = .45\linewidth]{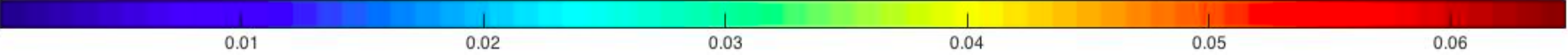}\qquad~~
\includegraphics[width = .45\linewidth]{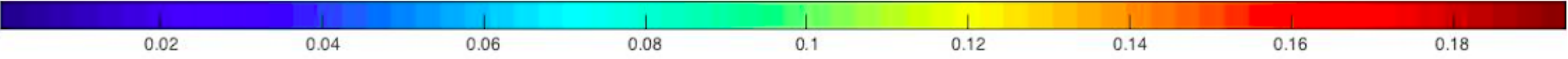}
\includegraphics[width = .80\linewidth]{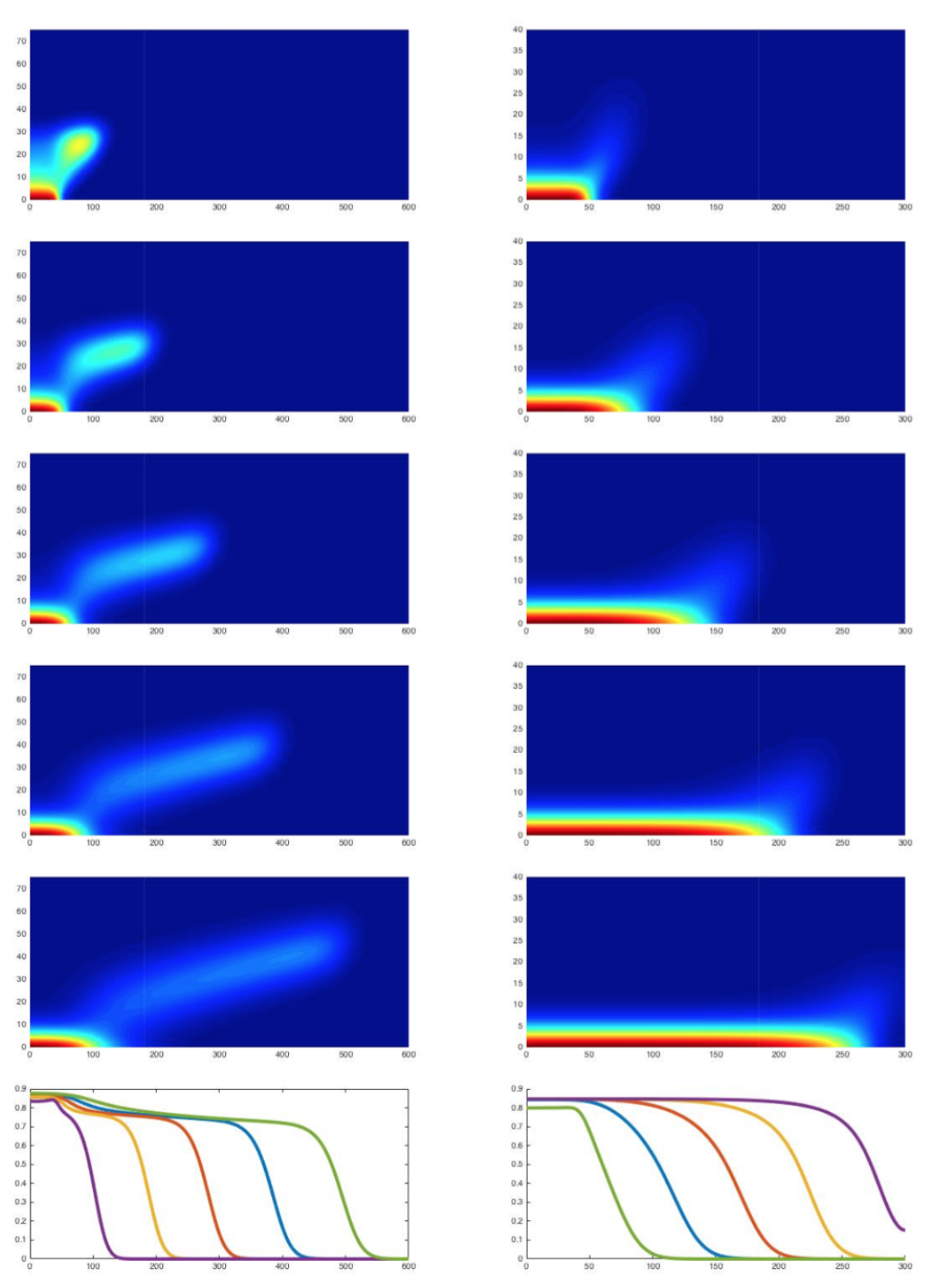}\\
\caption{Numerical simulations of the Cauchy problem of equation \eqref{eq:main} at a fixed time, in the phase space $\R\times\Theta$ at times (from top to bottom) $t=10$, $t=20$, $t=30$, $t=40$, $t=50$. Left column: $p=1/3$. Right column: $p=2/3$. Last line: evolution of $\rho$ at times $t=10$, $t=20$, $t=30$, $t=40$, $t=50$. Both exhibit  propagation at a super-linear rate.  The transient dynamics driving the acceleration are seen in the ``head'' -- the light diagonal line moving and up and to the right of the front.  This can be observed, e.g., in Step 2 of the proof of Proposition~\ref{prop:lowerboundacc}.}
\label{fig:Shape2}
\end{center}
\end{figure}

\subsection*{Acknowledgements}

EB is very grateful to the University of Sydney, where the present work has been initiated, for its hospitality.  The authors thank the University of Cambridge for its hospitality.  The authors thank warmly Vincent Calvez for early discussions about this problem, and for a careful reading of the manuscript.  CH thanks Alessandro Carlotto, Boaz Haberman, and Otis Chodosh for discussions about geometry and heat kernel estimates, which, while meant for earlier projects, found an application in this manuscript. Part of this work was performed within the framework of the LABEX MILYON (ANR- 10-LABX-0070) of Universit\'e de Lyon, within the program ``Investissements d'Avenir'' (ANR-11- IDEX-0007) operated by the French National Research Agency (ANR). In addition, this project has received funding from the European Research Council (ERC) under the European Unions Horizon 2020 research and innovation programme (grant agreement No 639638) and under the MATKIT starting grant. MHC and PSK were funded in part by the Australian Research Council (ARC) Discovery Project (DP160101597). We thank the anonymous referees for a close reading of the manuscript and very helpful suggestions.

\section{Proof of \Cref{thm:acceleration}: the acceleration regime}\label{sec:acceleration}

\subsection{The upper bound}

In this section, we present the proof of the upper bound in \Cref{thm:acceleration}, i.e. we show:
\begin{prop}\label{prop:upperboundacc}
	Under the assumptions of \Cref{thm:acceleration}, there exists $\overline a>0$ such that
	\[
			\lim_{t\to \infty}  \sup_{x\geq \eta_{\overline{a}}(t)^{3/2}} \sup_{\theta \in \Theta} n(t,x,\theta) = 0.
	\]
\end{prop}
We may then deduce the upper bound in \Cref{thm:acceleration} by recalling \eqref{eq:etaa_to_eta1}. To get an upper bound on the propagation, we use the linearization around zero of~\eqref{eq:main}:
\begin{equation*}
\begin{cases}
\overline n_t = \theta \overline n_{xx} + \overline n_{\theta\theta} + \left( 1 - m \right)  \overline n , &\qquad \text{on } \R^+ \times \R\times\Theta,\\
\overline n_\theta(\cdot,\underline\theta) = 0, & \qquad \text{on } \R^+ \times \R. 
\end{cases}
\end{equation*}

Indeed, solutions to~\eqref{eq:main} are sub-solutions to $\overline n$.  Hence, by the comparison principle, $n \leq \overline n$ and it suffices to prove Proposition~\ref{prop:upperboundacc} for $\overline n$. 
In order to bound $\overline n$, we now introduce a key quantity. It is the action associated to the Lagrangian corresponding to \eqref{eq:main}:
\[
	\zeta(t,x,\theta,y,\eta) = \inf\left\{
		\int_0^t \left(\frac{|\dot Z_1|^2}{4Z_2} + \frac{|\dot Z_2|^2}{4} + m(Z_2)\right)ds
			: Z(0) = (y,\eta), Z(t) = (x,\theta), Z\in C^{0,1}\left( [0,t] \right)
	 \right\}.
\]
In general, we simply write $\zeta(t,x,\theta)$ to mean $\zeta(t,x,\theta,0,\underline\theta)$. We seek a super-solution which asymptotically looks like $e^{t - \zeta}$.

%

\subsubsection*{A brief discussion of $\zeta$ and previous strategies}

We briefly describe why such a super-solution is expected and why the strategy from~\cite{BHR_Acceleration}, where we showed that $e^{t-\zeta}$ is a super-solution to the equation without trade-off, does not work in this setting.  When $m \equiv 0$, one may expect to follow the work of Evans and Souganidis~\cite{EvansSouganidis} to see that the front should follow the dynamics of the solution of $\zeta$ due to the fact that it is the solution to the equation 
\[
	\zeta_t + \theta|\zeta_x|^2 + |\zeta_\theta|^2 = 0.
\]
This is in the thin front limit with the scaling given in \cite{BCMetal} (also recalled in \cite{BHR_Acceleration}). In physical space, i.e.~without scaling, one can check that $e^{t-\zeta}$ is a super-solution to \eqref{eq:main} if and only if
\begin{equation}\label{eq:supersoln_condition}
0 \leq \theta \zeta_{xx} + \zeta_{\theta\theta}.
\end{equation}
When $m \equiv 0$, it can be checked that $\zeta$ satisfies this, see \cite{BHR_Acceleration}.

When the trade-off $m$ is present, it is much harder to check~\eqref{eq:supersoln_condition}.  Indeed, define the related Lagrangian by $L_m(Z,\dot Z) = \frac{1}{4} \left(\frac{|\dot Z_1|^2}{Z_2} + |\dot Z_2 |^2\right)  + m(Z_2)$ and let $Z$ be an optimal trajectory satisfying the Euler-Lagrange equations.  As mentioned above, it is well-known that $\zeta_t + \theta |\zeta_x|^2 + |\zeta_\theta|^2 = m(\theta)$.  On the other hand, a computation using the Euler-Lagrange equations for $Z$ yields
\[
	\zeta_{xx} = \int_0^1 (v_x(t) \cdot ((D^2 L)(v) v_x(t)) ds
		\qquad\text{ and }\qquad
	\zeta_{\theta\theta} = \int_0^1 (v_\theta(t) \cdot ((D^2L)(v)v_\theta(t)) ds,
\]
where $v(t) = (Z(t), \dot Z(t))$.  One can check that $D^2L_m = D^2L_0 + (\delta_{i2} \delta_{j2} m''(Z_2))_{ij}$, where $L_0$ is the Lagrangian with no trade-off term.  A straightforward computation shows that $L_0$ is nonnegative definite and hence, if $m'' \geq 0$ then $L_m$ is nonnegative definite; however, we expect that $m'' < 0$ holds for large $\theta$ since $m$ is sub-linear.  Even though $L_m$ may not be nonnegative definite, one may still hope that $\theta\zeta_{xx} + \zeta_{\theta\theta} \geq 0$.  However, the term involving $m''$ in $\theta \zeta_{xx} + \zeta_{xx}$ is
\[
	\int_0^1 m''(Z_2) \left[\theta |\partial_x Z_2|^2 + |\partial_\theta Z_2|^2 \right] ds.
\]
Since we cannot explicitly solve for $Z$, this integral is very difficult to estimate.  As such, we are unable to show that there are super-solutions of the form $e^{t - \zeta}$.  Similar difficulties arise when looking at $e^{at - b \zeta}$ for any choice of $a$ and $b$.

\subsubsection*{Relating $\overline n$ and $\zeta$}

We now explain how to relate $\overline n$ and $\zeta$. For this, we use results by Li and Yau~\cite[Corollary~3.2]{LiYau} on parabolic equations on Riemannian manifolds. In this paper, the authors derive bounds on the fundamental solutions of the heat equation with potential on general manifolds. In particular, they prove an upper bound after proving a suitable Harnack inequality.

We use their analysis to derive an upper bound on the fundamental solution of \eqref{eq:linearized}. To show that we can use their results, we view $\R\times\Theta$ as a Riemannian manifold with an appropriate metric $g$ so that the second order operator in \eqref{eq:linearized} is the Laplace-Beltrami operator on that manifold.

To this end, we need to remove an integrating factor of $\overline n$.  Let $r = 1+\underline \theta^{-2}$ and define $v = e^{-rt} \theta^{1/4} \overline n$. Notice that
\begin{equation}\label{eq:v_eqn}
	v_t = \theta v_{xx} + v_{\theta\theta} - \frac{1}{2\theta} v_\theta + v\left(1-m + \frac{5}{16 \theta^2} - r\right)
		\leq \theta v_{xx} + v_{\theta\theta} - \frac{1}{2\theta} v_\theta - mv.
\end{equation}

Let $G(t,x,y,\theta,\eta)$ be the fundamental solution to
\begin{equation}\label{eq:G_pde}
	G_t = \theta G_{xx} + G_{\theta\theta} - \frac{1}{2\theta} G_\theta - m G,
\end{equation}
on $\R\times \Theta$ with Neumann boundary conditions in $\theta$.  In order to obtain bounds on $G$, we wish to apply the bounds obtained in~\cite{LiYau}.  To this end, we define the metric
\begin{equation}\label{eq:metric}
	g = \begin{pmatrix}
		\frac{1}{\theta} & 0 \\ 0 & 1
		\end{pmatrix},
\end{equation}
and, denoting $\Delta_g$ as the Laplace-Beltrami operator associated to $g$, notice that $G$ satisfies
\[
	G_t = \Delta_g G - mG.
\]
In addition, one can check that the curvature of the manifold $(\R\times\Theta, g)$ is uniformly bounded. For the help of the reader, we include a discussion of all geometric issues in \Cref{sec:appendix}.

If $n_0$ were compactly supported, we would take $G(t+1,x,\theta,0,\underline\theta)$ as a super-solution.  Since $n_0$ is not compactly supported, we modify this and define $w$ as
\begin{equation}\label{eq:w}
	w(t,x,\theta)
		= C \sum_{\ell=0}^\infty G(t+1, x + \ell ,\theta,0,\underline\theta),
\end{equation}
where $C$ is chosen large enough that $w(0,\cdot) \geq \overline n_0$.  This is essentially the convolution of $G$ with the initial data, but this formulation is more convenient computationally in the sequel.  It follows that $e^{rt}\theta^{-1/4} w(t,x,\theta)$ is a super-solution to $\overline n$.  We point out that this super-solution is different from the one appearing in~\cite{BHR_Acceleration}.  In that article, because we had an explicit formula for the super-solution, we could extend the super-solution to the quadrant $\{x<0\}$ by ``forgetting'' the $x$-dependence in a way that the solution remained $\mathcal{C}^1$.  Here, we cannot do that since we have no way of verifying that such a construction is still a super-solution on the line $\{x = 0\}$.

With this set, we recall the following results by Li and Yau~\cite[Corollary~3.2]{LiYau}.

\begin{lemma}\label{lem:Li_Yau}
Let $t>t_0$ and $x \in \R$. For $\theta \leq \eta_{\gamma_\infty + 1}(t)$, there exists a constant $C$ such that
\[
	G(t,x,\theta,0,\underline\theta) \leq C \exp \left\{ C t - \frac{\zeta(t,x,\theta)}{2}\right\}.
\]
\end{lemma}
We refer to Appendix A for a discussion on how \Cref{lem:Li_Yau} follows from \cite[Corollary~3.2]{LiYau} as it is not immediate.  By~\eqref{eq:w} and \Cref{lem:Li_Yau}, a bound on $\overline n$ follows from a bound on $\zeta$.

\subsubsection*{A bound on $\zeta$ and the conclusion of the proof of the upper bound}

Our goal in this section is to derive a lower bound for $\zeta$ and then to use that to obtain a bound on $w$.  The first step is to prove the following lemma.
\begin{lemma}\label{lem:rho_bound}
Fix any $\overline a > 0$. Assume that $x \geq \eta_{\overline a}(t)^{3/2}$. There exists a constant $C>0$ such that
\[
	\zeta(t,x,\theta) \geq C^{-1}\min\left\{\overline a t\sqrt{ \frac{x}{\eta_{\overline a}(t)^{3/2}}}
		,\frac{x^2}{t}\right\}.
\]
\end{lemma}
We point out that each term in the minimum above has the correct or super-critical scaling.  In other words, taking $x > \eta_{\overline a}(t)^{3/2}$ in each term yields a term which is at least linear in time.
\begin{proof}
Thanks to the Euler-Lagrange equations, we know that the infimum defining $\zeta$ is actually a minimum and the minimizer satisfies
\begin{equation*}
\frac{d}{ds} \left( \frac{\dot Z_1}{2 Z_2} \right) = 0,\qquad \frac{d}{ds} \left( \frac{\dot Z_2}{2}\right) = m'(Z_2) - \frac{\vert\dot  Z_1 \vert^2}{4 Z_2^2}.
\end{equation*}
Hence there exists $\alpha$ depending on $t,x,\theta$ such that $\dot Z_1 = 2\alpha Z_2$.
Integrating this, we find $\alpha = x/(2 \int_0^t Z_2(s) ds)$.
Combining this with the identity $\dot Z_1 = 2\alpha Z_2$ and the definition of $\zeta$, we obtain
\begin{equation}\label{eq:rhosimple}
	\zeta(t,x,\theta) = \frac{x^2}{4 \int_0^t Z_2 \, ds}  + \int_0^t \left( \frac{|\dot Z_2 |^2}{4} + m(Z_2 ) \right) ds.
\end{equation}
We define $M_\theta = \max_{s \in [0,t]} Z_2(s)$. Set $B_+ = \{s: Z_2(s) \geq \theta_d\}$ and $B_- = [0,t]\setminus B_+$.  Recall that $m(\theta)/\theta$ is decreasing for $\theta\geq \theta_d$.    Define also $s_0$ such that $Z_2(s_0) = M_\theta$.  There are two cases.

\medskip

\noindent{\bf \# First case: $\int_{B_-} Z_2(s) ds \leq \int_{B_+} Z_2(s) ds$.}

\medskip

In this case, $M_\theta \geq \theta_d$. Indeed, if not, $B_+$ is an empty set which would yield a contradiction. Notice that $m(Z_2)/Z_2 \geq m(M_\theta)/M_\theta$ on $B_+$. Using this and applying Young's inequality gives
\begin{equation}\label{eq:zetabound}
\begin{split}
	\zeta(t,x,\theta) &\geq \frac{x^2}{4\int_0^t Z_2(s)ds} + \int_0^t m(Z_2(s))ds
		\geq \frac{x^2}{4\left(\int_{B_-} Z_2(s) ds + \int_{B_+} Z_2(s)ds\right)} + \int_{B_+} m(Z_2(s))ds\\
		&\geq \frac{x^2}{8 \int_{B_+} Z_2(s)ds} + \frac{m(M_\theta)}{M_\theta} \int_{B_+} Z_2(s)ds
		\geq x \left( \frac{m(M_\theta)}{2M_\theta} \right)^\frac{1}{2}.
\end{split}
\end{equation}

We are now ready to bound $\zeta$ when $M_\theta$ is small, i.e. when $M_\theta \leq \eta_{\overline a}(t)$.
%
%
Recall that when $\theta \geq \theta_d$, $m(\theta)/\theta$ is decreasing.  Due to this and \eqref{eq:phi_to_m}, together with the fact that $M_\theta \geq \theta_d$, we find 
\[
		\frac{\overline a t}{\sqrt2} \left(\frac{x}{\eta_{\overline a}(t)^{3/2}} \right)^\frac12
		\leq \frac{\overline a t}{\sqrt2} \frac{x}{\eta_{\overline a}(t)^{3/2}}
		\leq x \left( \frac{m(\eta_{\overline a}(t))}{2\eta_{\overline a}(t)} \right)^\frac{1}{2}
		\leq x \left( \frac{m(M_\theta)}{2M_\theta} \right)^\frac{1}{2}
		\leq \zeta(t,x,\theta).
\]
This concludes the estimate of $\zeta$ is this sub-case.

%
%

Now consider the case when $M_\theta$ is large, i.e.~$M_\theta \geq \eta_{\overline a}(t)$.  Young's inequality implies that
\begin{equation}\label{eq:zetabound2}
\zeta(t,x,\theta) \geq \int_0^{s_0} \left( \frac{|\dot Z_2|^2}{4} +  m(Z_2) \right) ds
		\geq \int_0^{s_0} \dot Z_2 \sqrt{m(Z_2)} ds
		= \int_0^{s_0} \partial_s (\Phi(Z_2(s))) ds
		= \Phi(M_\theta).
\end{equation}
Combining \eqref{eq:zetabound} and~\eqref{eq:zetabound2} and Young's inequality, we obtain, recalling also \eqref{eq:phi_to_m},
\begin{equation*}
\begin{split}
	\zeta(t,x,\theta) &= \frac{\zeta(t,x,\theta)}{2} + \frac{\zeta(t,x,\theta)}{2}
		\geq \frac12 x \Big( \frac{m(M_\theta)}{2M_\theta} \Big)^\frac{1}{2} + \frac{\Phi(M_\theta)}{2}
		\\ &\geq \left( x \Big( \frac{m(M_\theta)}{2M_\theta} \Big)^\frac{1}{2} \Phi(M_\theta)\right)^\frac12
		\geq \Bigg( \frac{x M_\theta^\frac{1}{2} m(M_\theta)}{2^\frac12 D_m} \Bigg)^\frac12.
\end{split}
\end{equation*}
From this and the fact that $M_\theta \geq \eta_{\overline a}(t)$, we conclude that
\begin{equation}
	\zeta(t,x,\theta)
		\geq \left( \frac{x \eta_{\overline a}(t)^\frac{1}{2} \frac{\Phi(\eta_{\overline a}(t))^2}{\eta_{\overline a}(t)^2} }{2^\frac12 D_m} \right)^\frac12 
		\geq \left( \frac{x \eta_{\overline a}(t)^{-\frac{3}{2}}}{2^\frac12 D_m} \right)^\frac12 \overline{a}t.
\end{equation}

\medskip

\noindent{\bf \# Second case: $\int_{B_-} Z_2(s) ds > \int_{B_+} Z_2(s) ds$.}

\medskip

In this case, it follows that
\[
	\int_0^t Z_2(s)ds
		\leq \int_{B_-}Z_2(s)ds + \int_{B_+}Z_2(s)ds
		\leq 2 \int_{B_-} Z_2(s) ds
		\leq 2 \theta_d t.
\]
Hence, we finish by plugging this into~\eqref{eq:rhosimple} to obtain the bound $\zeta(t,x,\theta) \geq x^2/(8\theta_d t)$.
\end{proof}

With \Cref{lem:rho_bound} in hand, we are in a position to finish the proof of the upper bound of $\overline n$.
\begin{proof}[{\bf Proof of Proposition~\ref{prop:upperboundacc}}]

Assume first that $\theta \geq \eta_{\gamma_\infty + 1}(t)$. Proposition~\ref{prop:upperboundacc} follows by combining \eqref{eq:Q_asymptotics} with the fact that a function of the form $Q(\theta)e^{\gamma_\infty  t}$ is a super-solution to \eqref{eq:main}.

Thus, we may assume that $\theta \leq \eta_{\gamma_\infty + 1}(t)$.
Fix $x \geq \eta_{\overline a}(t)^{3/2}$ with $\overline a$ to be determined.  Using the definition of $w$ along with \Cref{lem:Li_Yau}, we have that
\[
	w(t,x,\theta)
		\leq C e^{Ct} \sum_{\ell=0}^\infty \exp\left\{ - \zeta(t,x+\ell,\theta)\right\}.
\]
We obtain a very rough bound on $w$ in the following way:
\begin{equation}\label{eq:upper_estimate}
\begin{split}
	w(t,&x,\theta) \leq \frac{Ce^{Ct}}{t} \sum_{\ell = 0}^\infty \exp\Big\{ - \frac{1}{2C} \min\Big(\overline{a} t\sqrt{ \frac{x+\ell}{\eta_{\overline a}(t)^{3/2}}},\frac{(x+\ell)^2}{t}\Big)\Big\} \\
		&\leq \displaystyle\frac{Ce^{Ct}}{t} \sum_{\ell = 0}^\infty 
		 \Big(e^{ - \frac{\overline a t}{2C}\sqrt{\frac{x+\ell}{\eta_{\overline a}(t)^{3/2}}}}
			+e^{-\frac{(x+\ell)^2}{2Ct}}\Big)
		\leq C e^{Ct} \Big(\frac{\eta_{\overline a}(t)^{3/2}}{\overline a^2 t^2} e^{ - \frac{\overline a t}{C}\sqrt{\frac{x}{\eta_{\overline a}(t)^{3/2}}}}
			+ \frac{e^{-\frac{x^2}{Ct}}}{\sqrt{t}}\Big).
\end{split}
\end{equation}
For $t\geq 1$, it is clear from its definition that $\eta_{\overline a}(t) \gg t^{2/3}$ and, hence, that $x^2/t \gg t$.  On the other hand, since $x \geq \eta_{\overline a}(t)^{3/2}$, the exponent of the first term may be bounded below by $\overline a t / C$.  Using this and choosing $\overline a$ big enough, we see that the right hand side of~\eqref{eq:upper_estimate} may be bounded by $C_{\overline a}e^{-(r+1) t}$.

We now apply this bound to finish the claim.  Recalling the definition of $v$ and recalling that $w$ is a super-solution to $v$, if $x \geq \eta_{\overline a}(t)^{3/2}$,
\[
	\overline n(t,x,\theta)
		= \theta^{-1/4} e^{rt} v(t,x,\theta)
		\leq \underline\theta^{-1/4} e^{rt} w(t,x,\theta)
		\leq C \underline\theta^{-1/4}e^{-t}.
\]
\end{proof}

\subsection{The lower bound}

In order to finish the proof of \Cref{thm:acceleration}, we now need to prove the lower bound, i.e.~ we show:

\begin{prop}\label{prop:lowerboundacc}
Under the assumptions of \Cref{thm:acceleration}, there exists $\underline n, \underline{a}>0$ such that
	\[
		\liminf_{t \to \infty} \inf_{|x| \leq \eta_{\underline{a}}(t)^{3/2}} n(t,x,\underline\theta) \geq \underline n.
	\]
\end{prop}

\subsubsection*{The moving Dirichlet ball sub-solution}

As we mentioned in the introduction, to prove spreading, the idea is to construct sub-solutions to the linearized problem with Dirichlet boundary conditions on a moving boundary of a growing domain $\mathcal{E}(t)$.  Then we use them to deduce a lower bound on the solution of the nonlinear problem.  

When building this solution, the growth/decay rate depends on the speed of $\mathcal{E}(t)$ -- the faster $\mathcal{E}$ moves, the smaller the growth rate with the growth rate tending to negative infinity as the speed of $\mathcal{E}$ tends to infinity.  Thus, the goal is to balance two competing forces: when $\mathcal{E}$ is in the large $\theta$ region the correlation between the speed of $\mathcal{E}$ and the growth rate of $\mathcal{E}$ is weakest, i.e.~in the large $\theta$ region, $\mathcal{E}$ can move at a much faster rate in the $x$ direction than if it were in the small $\theta$ region with the same effect on the growth rate; on the other hand, when $\mathcal{E}$ is in the large $\theta$ region, the trade-off term $m$ is extremely strong causing the growth rate to be very negative.  In view of this, our goal is to find a trajectory that takes advantage of both the fast movement when $\mathcal{E}$ is in the large $\theta$ region and the positive growth rate when $\mathcal{E}$ is in the small $\theta$ region.

We thus state a lemma regarding sub-solutions on moving, growing ellipses. This lemma is very similar to \cite[Lemma~4.1]{BHR_Acceleration}.  Before we state the lemma, we define a piece of notation.  For any given trajectory $t \mapsto \left( X(t) , \Theta(t) \right) \in \R\times\Theta$ and $\Lambda > 0$, we denote 
\begin{equation}\label{eq:ellipse_defn}
 	\mathcal{E}_{t,\Lambda}^{\left( X , \Theta \right)}
		\stackrel{\rm def}{=}
		\Big\{
			(x,\theta) \in \R \times \overline\Theta : \frac{\vert x - X(t) \vert^2}{\Theta(t)} + \vert \theta - \Theta(t) \vert^2 \leq \Lambda^2
		\Big\}.
\end{equation}

\begin{lemma}\label{lem:movement}
Let $T > T_0$ and let $(X_T(t), \Theta_T(t)) \in \mathcal{C}^2([T_0,T])$ be a trajectory.  Fix constants $\Lambda$, $r$, and $\delta$. There exists $\epsilon_\Lambda > 0$, depending only on $\Lambda$ and which is decreasing in $\Lambda$, $T_\delta$, depending only on $\delta$, and $\Lambda_0$ such that if $\Lambda \geq \Lambda_0$, $T - T_0 \geq T_\delta$ and
\begin{equation}\label{eq:lemma_condition}
\begin{split}
\forall t\in [T_0,T], \qquad	\left|\frac{1}{\Theta_{T}(t)}\right|  +  \left|\frac{1}{2} \frac{\dot\Theta_{T}(t)}{\Theta_{T}(t)}\right| + \left|\frac{\dot X_{T}(t)}{\Theta_{T}(t)^{3/2}}\right| \leq \epsilon_\Lambda,
\end{split}
\end{equation}
then there exists a function $\underline n$ which satisfies
\begin{equation}\label{eq:subsolE}
\begin{cases}
\underline n_t - \theta \underline n_{xx} - \underline n_{\theta\theta} \leq  (1 - r)\underline n, & \qquad \text{on } [T_0,T] \times \mathcal{E}_{t,\Lambda}^{\left( X_T , \Theta_T \right)}, \smallskip\\
\underline n = 0,  &\qquad \text{on } [T_0,T] \times \partial\mathcal{E}_{t,\Lambda}^{\left( X_T , \Theta_T \right)}, 
\end{cases}
\end{equation}
and such that $\underline n(T_0,\cdot) \leq \delta$ on  $\mathcal{E}_{T_0,\Lambda}^{\left( X , \Theta \right)}$ and
\begin{equation}\label{eq:constantsliding}
\begin{split}
	\underline n(T,\cdot) \geq &\delta C_\Lambda \exp\Big\{ - \frac{\Lambda}{2} \max_{t\in[0,T]} \Big( |\dot\Theta_{T}(t)| + \frac{|\dot X_{T}(t)|}{\sqrt{\Theta_{T}(t)}} \Big)\\
	&~~~~~ - \int_{T_0}^{T} \Big[ r + \frac{\dot{X}_{T}^2}{4\Theta_{T}} + \frac{\dot{\Theta}_{T}^2}{4} + 
			\frac{|\ddot{X}_{T}|\Lambda}{2} + \frac{|\dot{X}_{T}\dot{\Theta}_{T}|\Lambda}{4\Theta_{T}}  
			+ \frac{\dot{X}_{T}^2 \Lambda}{4\Theta_{T}^2} + \frac{\Lambda|\ddot{\Theta}_{T}|}{2} \Big] dt \Big\}.
\end{split}
\end{equation}
on $\mathcal{E}_{T,\Lambda/2}^{\left( X , \Theta \right)}$, where  $C_\Lambda$ is a positive constant depending only on $\Lambda$.
\end{lemma}

\begin{proof}
Since the lemma is very similar to the one in \cite{BHR_Acceleration}, we present a streamlined version of the proof. We recall that the idea of the proof is to suitably re-scale the equation and then use careful time-dependent spectral estimates.

First note that, without loss of generality, we may set $T_0 = 0$.  To construct the desired sub-solution $\underline n$, we first go into the moving frame, and rescale the spatial variable:
\begin{equation}\label{eq:v}
\underline{n}(t,x,\theta) = \underline{\tilde n}\left( t, \frac{x - X_T}{\sqrt{\Theta_T}}, \theta - \Theta_T\right),~~~
y = \frac{x - X_T}{\sqrt{\Theta_T}}, ~~\text{ and }~~ \eta = \theta - \Theta_T.
\end{equation}
Then plugging this into \eqref{eq:subsolE} yields  
\begin{equation}\label{eq:v_inequality}
\begin{cases}
	\displaystyle\underline{\tilde n}_t  - \Big(\frac{y}{2} \frac{\dot\Theta_T}{\Theta_T} + \frac{\dot{X}_T}{\sqrt{\Theta_T}} \Big)\underline{\tilde n}_y  - \dot\Theta_T \underline{\tilde n}_\eta \leq \Big(1 + \frac{\eta}{\Theta_T}\Big) \underline{\tilde n}_{yy} + \underline{\tilde n}_{\eta\eta} + (1 - r) \underline{\tilde n}, &\text{on } \mathcal{B}_\Lambda,\\
\underline{\tilde n}(t,\cdot) = 0, &\text{on } \partial\mathcal{B}_\Lambda.	
	\end{cases}
\end{equation}
Here, $\mathcal{B}_\Lambda\stackrel{\rm def}{=}B_\Lambda(0,0)$ is a ball of radius $\Lambda$ centered at $(y,\eta) = (0,0)$.   
%
As in~\cite{BHR_Acceleration}, the next step is to remove a suitable exponential:
\[
w(t,y,\eta) = \exp\Big\{ \frac12\big(y\,\dot{X_T}\, \Theta_T^{-1/2} + \dot\Theta_T\, \eta \big) + g(t)\Big\} \underline{\tilde n}(t,y,\eta),
\]
where $g$ is defined below. Using~\eqref{eq:v_inequality}, we see that $w$ must satisfy the inequality
\begin{equation}\begin{split}
	w_t - &\underbrace{\Big(\frac{y}{2}\frac{\dot\Theta_T}{\Theta_T} - \frac{\dot{X_T}}{\Theta_T^{\frac32}}\eta\Big)}_{\stackrel{\text{def}}{=} A} w_y
		\leq \underbrace{\Big(1 + \frac{\eta}{\Theta_T}\Big)}_{\stackrel{\text{def}}{=} D} w_{yy} + w_{\eta\eta}\\
			& + w\Big( 1 - r - \frac{\dot{X}_T^2}{4\Theta_T} - \frac{\dot{\Theta}_T^2}{4} + 
			\Big( \frac{\ddot{X}_T}{2\sqrt{\Theta_T}} - \frac{\dot{X}_T\dot{\Theta}_T}{4\Theta_T^{\frac32}} \Big) y 
			+ \Big( \frac{\dot{X}_T^2}{4\Theta_T^2} + \frac{\ddot{\Theta}_T}{2} \Big) \eta + g' \Big).
\end{split}
\label{eq:trajectories}
\end{equation}

We point out that each of the perturbative terms $A$ and $D - 1$ tend to zero as $\Lambda$ tends to infinity by the hypothesis \eqref{eq:lemma_condition}. We get rid of the supplementary terms in the growth part, by setting
\begin{equation*}
g(t) := \int_0^{t} \left[r + \frac{\dot{X}_{T}^2}{4\Theta_{T}} + \frac{\dot{\Theta}_{T}^2}{4} + 
			\frac{|\ddot{X}_{T}|\Lambda}{2 \sqrt{\Theta_T}} + \frac{|\dot{X}_{T}\dot{\Theta}_{T}|\Lambda}{4\Theta_{T}^{\frac32}}  
			+ \frac{\dot{X}_{T}^2 \Lambda}{4\Theta_{T}^2} + \frac{\Lambda|\ddot{\Theta}_{T}|}{2} \right] dt'.
\end{equation*}

%

Returning to the original variables, if $w$ satisfies
\begin{equation}\label{eq:key_differential_inequality}
	w_t - A w_y \leq D w_{yy} + w_{\eta\eta} + w,
\end{equation}
then $v$ satisfies the desired differential inequality.  
With this in mind, we seek to construct $w$ satisfying~\eqref{eq:key_differential_inequality} that has the desired bounds.

We define $w$ using the principal eigenfunction of the operator
\begin{equation}\label{eq:L}
	\mathcal{L}_t \stackrel{\text{def}}{=} A \partial_y + D \partial_{yy} + \partial_{\eta\eta}.
\end{equation}
To this end, for each $t \in [0,T]$, define $(\lambda^t,\phi^t)$ to be the principal Dirichlet eigenelements of $\mathcal{L}_t$ (depending on $t$ as a parameter) in the ball $\mathcal{B}_\Lambda$,
with the normalization $\|\phi^t\|_{L^\infty \left(\mathcal{B}_\Lambda\right)} = 1.$
We define $w(t,y,\eta) = \delta \phi^t(y,\eta)$. 
Then, we  have
\begin{equation*}
\left( \partial_t - \mathcal{L}_t \right) w = \frac{\partial_t \phi^t}{\phi^t} w - \lambda_t w,
\end{equation*}
and
$w$ satisfies~\eqref{eq:key_differential_inequality} if we are able to ensure that 
\begin{equation*}
	\frac{\partial_t \phi^t}{\phi^t}  - \lambda_t \leq 1,
\end{equation*}
which follows from~\cite[Lemmas 5.1 and 5.2]{BHR_Acceleration} along with~\eqref{eq:lemma_condition}.

Again, due to~\cite[Lemma 5.1]{BHR_Acceleration}, $\phi^t$ converges uniformly to the principal Dirichlet eigenfunction of $B_\Lambda$ as $\epsilon_\Lambda$ tends to zero.  Hence, by choosing $\epsilon_\Lambda$ small enough, we obtain a constant $C_\Lambda$ such that
\[
	\delta C_\Lambda
		\leq \delta \min_{\mathcal{B}_{\Lambda/2}} \phi^T
		=  \min_{\mathcal{B}_{\Lambda/2}} w(T,\cdot).
\]
By undoing the change of variables and using the relationship between $u$ and $w$, we are finished.
\end{proof}

\subsubsection*{The proof of Proposition~\ref{prop:lowerboundacc}: moving along a particular trajectory}

\begin{proof}[{\bf Proof of Proposition~\ref{prop:lowerboundacc}}]

Here, we describe the trajectory on which we apply \Cref{lem:movement}. It consists of three steps. First, we move mass upwards. Of course, this mass reaches a place where the death rate is highly negative due to the strength of the trade-off. This movement is justified by the second part of our trajectory, where we are able to move forward in space with a very high velocity, since the space diffusivity is very high in this zone. However, due to the strong trade-off, the mass at the end of this second step is extremely small. It is thus mandatory to move down to small traits again, to reach a zone where it is possible for the population to grow. Finally, in this region, we grow the population to order one in the last step. We show this trajectory in \cref{fig:trajectory}.

\begin{figure}[h]
\begin{center}
\includegraphics[width = .87\linewidth]{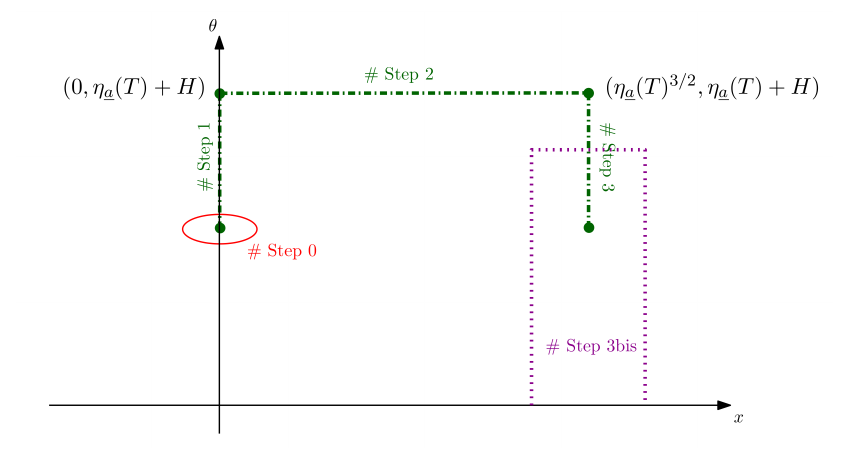}
\caption{The trajectory and the different steps.}
\label{fig:trajectory}
\end{center}
\end{figure}



This strategy of moving along trajectories in the phase space $\R\times\Theta$ was used in \cite{BHR_Acceleration} to prove a precise estimate on acceleration in the case $m\equiv 0$. There, it is shown that taking trajectories that come from a related Hamilton-Jacobi equation gives a sharp bound for the local cane toads equation. Here, we are not able to explicitly solve the Euler-Lagrange equations to derive the expression of optimal Hamilton-Jacobi trajectories due to the presence of the trade-off.

Fix any large time $T>0$.  We now describe quantitatively the three steps mentioned earlier.

\smallskip

{\bf \# Step 0: Initialization}

\smallskip

By the maximum principle, $n$ must be positive everywhere at time $t=1$.  As such, we assume without loss of generality that $n$ is positive everywhere initially.

\smallskip

{\bf \# Step 1 : Moving up}

\smallskip

First, fix constants $\underline a$, $A$, $\Lambda_1$, and $H$ to be determined later, where we eventually choose $\underline{a}$ to be small and $A$, $\Lambda_1$, and $H$ to be large.  We move mass upwards.  To this end, define the trajectory 
\begin{equation}\label{eq:upward_trajectory}
	\left(X_1(t),\Theta_1(t)\right) \stackrel{\rm def}{=} \left( 0 ,  (c_1 t)^2 + H\right),
\end{equation}
for $t\in[0,T_1]$, where we have chosen $c_1 = T/A\eta_{\underline a}(T)^{3/2}$ and $T_1 = A \eta_{\underline a}(T)^2 / T$.
We point out that \eqref{eq:phi_to_m} implies that $\lim_{T \to \infty}\eta_{\underline a}(T)/T =0$. Hence $T_1 \ll T$ when $T$ is large.  Notice that the definition of the trajectory \eqref{eq:upward_trajectory} implies that $\Theta_1(T_1) = \eta_{\underline a}(T) + H$.  By choosing $H$ sufficiently large, depending on $\Lambda_1$, we have 
\begin{equation}
\begin{split}
\forall t\in [0,T_1], \qquad	\Big|\frac{1}{\Theta_{1}}\Big|  +  \Big|\frac{1}{2} \frac{\dot\Theta_{1}}{\Theta_{1}}\Big| + \Big|\frac{\dot X_{1}}{\Theta_{1}^{3/2}}\Big| \leq \epsilon_\Lambda,
\end{split}
\end{equation}
and we have met the conditions of \Cref{lem:movement}, which we apply in the sequel.  

We emphasize two points coming from the definition of the trajectory.  First, it is crucial to make the trajectory start at $\theta = H$ in order to satisfy that $\Theta_1(t)^{-1} \leq \epsilon_\Lambda$ for all $t$.  Second, while trajectories that are linear in time would be simpler, this is not possible here as we require that $\dot\Theta_1(t)/(2\Theta_1(t)) \leq \epsilon_\Lambda$.  A linear-in-time trajectory which ends at $O(\eta_{\underline a}(T))$ would not satisfy this inequality for small times.

We seek a sub-solution of $n$ on the set $\mathcal{E}_{t,\Lambda_1}^{\left( X_1 , \Theta_1 \right)}$. To this end, first notice that from Proposition~\ref{prop:uniform_upper_bound}, we know that there exists $\overline\rho$ such that $\rho \leq \overline \rho$ on $\R^+ \times \R\times\Theta$.  In addition, since $m$ is monotonic by \Cref{hyp:m}, $m(\theta) \leq m(\eta_{\underline a}(T) + H + \Lambda_1)$ on $\mathcal{E}_{t,\Lambda_1}^{\left( X_1 , \Theta_1 \right)}$.  Hence, any $\underline n$ solving 
\[\begin{cases}
	\underline n_t - \theta \underline n_{xx} - \underline n_{\theta\theta} \leq  \left(1 - \overline \rho - m \left(\eta_{\underline a}(T) + H + \Lambda_1 \right) \right)\underline n, & \qquad \text{on } [0,T_1] \times \mathcal{E}_{t,\Lambda_1}^{\left( X_1 , \Theta_1 \right)},\smallskip\\
	\underline n = 0,  &\qquad \text{on } [0,T_1] \times \partial\mathcal{E}_{t,\Lambda_1}^{\left( X_1 , \Theta_1 \right)}, 
\end{cases}\]
is a sub-solution to $n$.  Define
\[
	\delta_1^{\Lambda_1,H} \stackrel{\rm def}{=} \inf_{\mathcal{E}_{0,\Lambda_1}^{\left( X_1 , \Theta_1 \right)}} n(0,\cdot).
\]
Then, applying \Cref{lem:movement} with starting time $T_0 = 0$ and final time $T_1$, $r = \overline \rho + m(\eta_{\underline a}(T) + H + \Lambda_1)$, $\Lambda = \Lambda_1$, and $\delta = \delta_1$, we obtain $\underline{n_1}$  such that $\underline{n_1}(0,\cdot) \leq n(0,\cdot)$ on $\mathcal{E}_{0,\Lambda}^{\left( X_1 , \Theta_1 \right)}$.

We may then use the comparison principle to compare $\underline n_1$ and $n$ on $\mathcal{E}_{T_1,\Lambda_1/2}^{(X_1,\Theta_1)}$ at $t=T_1$.  Using~\eqref{eq:constantsliding}, we obtain
\begin{equation}\label{eq:size_after_sliding1}
\begin{split}
	n(T_1, x, \theta)
		&\geq \delta_1^{\Lambda_1,H}
			C_{\Lambda_1}
			\exp\left\{ - \Lambda_1 c_1^2 T_1 - \left( r T_1 + c_1^4 \frac{T_1^3}{3} + \Lambda_1 c_1^2 T_1 \right) \right\}\\
		&= \delta_1^{\Lambda_1,H}
			C_{\Lambda_1}
			\exp\left\{ -  2\Lambda_1 c_1^2 T_1 - r T_1 - \frac{c_1^4T_1^3}{3} \right\}\\
		&= \delta_1^{\Lambda_1,H}
			C_{\Lambda_1}
			\exp\left\{ - \frac{2\Lambda_1 T}{A \eta_{\underline a}(T)} - \frac{A \eta_{\underline a}(T)^2}{T} \left( \overline\rho + m(\eta_{\underline a}(T) + H + \Lambda_1) \right) - \frac{T}{3A}\right\}\\
		&= \delta_1^{\Lambda_1,H}
			C_{\Lambda_1}
			\exp\left\{ - T \left(\frac{1}{3A} + \frac{2\Lambda_1}{A \eta_{\underline a}(T)} + \frac{A \eta_{\underline a}(T)^2}{T^2} \left( \overline\rho + m(\eta_{\underline a}(T) + H + \Lambda_1) \right) \right)\right\}\\
		&\geq \delta_1^{\Lambda_1,H} C_{\Lambda_1} \exp\left\{ - T\left(\underline aA + \frac{2}{3A}\right)\right\}
		\geq \delta_1^{\Lambda_1,H} C_{\Lambda_1} \exp\left\{ - \frac{\gamma_\infty}{10}T\right\},
\end{split}
\end{equation}
for all $(x,\theta) \in \mathcal{E}_{T_1,\Lambda_1/2}^{(X_1,\Theta_1)}$. We comment on how we have obtained the second-to-last and last inequalities. 
The second-to-last is due to the following facts. First, since $\lim_{T \to \infty}\eta_{\underline a}(T) = + \infty$, one can take $T$ sufficiently large depending only on $\underline{a}$, and $\Lambda_1$, such that $2\Lambda_1/\eta_{\underline a}(T) \leq 1/3$. Second, since $\lim_{T \to \infty}\eta_{\underline a}(T)/T =0$, if $T$ is sufficiently large, depending only on $a$, then $\eta_{\underline a}(T)^2 \overline\rho /T^2 \leq \underline{a}/2$. Finally, if $T$ is taken sufficiently large so that $\eta_{\underline a}(T) \geq \Lambda_1 + H$, then we  have 
\begin{equation*}
	\frac{\eta_{\underline a}(T)^2}{T^2} m(\eta_{\underline a}(T) + H + \Lambda_1)
		\leq \frac{\eta_{\underline a}(T)^2}{T^2} m(2\eta_{\underline a}(T))
		\leq 2\frac{\eta_{\underline a}(T)^2}{T^2} D_m^2 \left( \frac{\Phi(\eta_{\underline a}(T))}{\eta_{\underline a}(T)} \right)^2
		= 2 D_m^2 \underline{a}^2 \leq \frac{\underline{a}}{2},
\end{equation*}
when $\underline{a}$ is taken sufficiently small, after using the monotonicity of $m$ and \eqref{eq:phi_to_m}.  The last inequality of \eqref{eq:size_after_sliding1} follows from choosing $A$ sufficiently large and then $\underline{a}$ sufficiently small.

\smallskip

{\bf \# Step 2 : Moving right}

\smallskip

We now build a new moving bump sub-solution starting at $t = T_1$ and $(0,\eta_{\underline a}(T) + H)$ which moves mass to the right at a high speed.  It is important to note that $\underline{n_1}$ and $\underline{n_2}$ may not be concatenated to be a smooth sub-solution.  Instead, we ``fit'' our new sub-solution $\underline{n_2}$ underneath $n$ at time $T_1$ using the lower bound we obtained from $\underline{n_1}$.

For $t\in[T_1, 2T_1]$, we define the trajectory
\[
	(X_2(t), \Theta_2(t)) \stackrel{\rm def}{=} (c_2 (t-T_1), \eta_{\underline a}(T) + H),
\]
where $c_2$ is chosen so $A \sqrt{\eta_{\underline a}(T)} c_2 = T$.  We point out that $X_2(2T_1) = \eta_{\underline a}(T)^{3/2}$, by our choice of $c_2$ and $T_1$, and that $(X_2(T_1), \Theta_2(T_1)) = (X_1(T_1), \Theta_1(T_1))$.

Set $\Lambda_2 = \Lambda_1/2$.  It is easy to check that this trajectory satisfies~\eqref{eq:lemma_condition} for $T$ sufficiently large since $c_2 \eta_{\underline a}(T)^{-3/2} = T/ (A \eta_{\underline a}(T)^2)$ tends to zero.  Letting $r$ be as before, $\Lambda = \Lambda_2$, and $\delta_2 = C_{\Lambda_1} \delta_1^{\Lambda_1,H} \exp\left\{ - \gamma_\infty T/10\right\}$, we apply \Cref{lem:movement} again to obtain $\underline{n_2}$ satisfying~\eqref{eq:subsolE} with
\begin{equation}\label{eq:subsoln_sliding2}
	\underline{n_2}(T_1,x,\theta) \leq \delta_2 \leq n(T_1,x,\theta)
\end{equation}
for all $(x,\theta) \in \mathcal{E}_{T_1,\Lambda_2}^{(X_2,\Theta_2)}$.  The second inequality above comes from~\eqref{eq:size_after_sliding1} and the definition of $\delta_2$.  Since~\eqref{eq:subsoln_sliding2} and~\eqref{eq:subsolE} guarantee that $\underline{n_2} \leq n$, then~\eqref{eq:constantsliding} gives us that, for all $(x,\theta) \in \mathcal{E}_{2T_2,\Lambda_2/2}^{(X_2,\Theta_2)}$,
\[\begin{split}
	n(2T_1, x,\theta)
		&\geq \underline{n_2}(2T_1,x,\theta)
		\geq \left(C_{\Lambda_1} \delta_1^{\Lambda_1,H} e^{- \frac{\gamma_\infty T}{10}}\right) \times \\
		&\qquad \left(C_{\Lambda_2} \exp\left\{ - \frac{3\Lambda_2}{4}\frac{T}{A\eta_{\underline a}(T)} - \frac{A \overline \rho \eta_{\underline a}(T)^2}{T} - \frac{A m(\eta_{\underline a}(T) + H + \Lambda_1) \eta_{\underline a}(T)^2}{T}- \frac{T}{4A} \right\}\right).
\end{split}\]
Exactly as before, so long as $\underline{a}$ is sufficiently small and $A$ and $T$ are sufficiently large, we may estimate the the bottom line of the equation above exactly as in~\eqref{eq:size_after_sliding1} to obtain the bound
\[
	n(2T_1, x,\theta)
		\geq C_{\Lambda_1} \delta_1^{\Lambda_1,H}  e^{- \frac{\gamma_\infty T}{10}}e^{- \frac{\gamma_\infty T}{10}}
		\geq C_{\Lambda_1} \delta_1^{\Lambda_1,H} e^{- \frac{\gamma_\infty T}{5}}.
\]
Here we also used the relationship between $\Lambda_1$ and $\Lambda_2$.

\smallskip

{\bf \# Step 3: Moving down}

\smallskip

The argument here is almost exactly as in the last two steps so we only briefly outline it.  Define
\[
	(X_3(t), \Theta_3(t)) = (\eta_{\underline a}(T)^{3/2}, \Theta_1(3T_1 - t))
\]
for all $t\in[2T_1, 3T_1]$.  We notice a few things.  First, $(X_3(2T_1),\Theta_3(2T_1)) = (X_2(2T_1),\Theta_2(2T_1))$.  Second, up to translation, $\Theta_3$ is the time reversal of $\Theta_1$.  Third, $(X_3(3T_1), \Theta_3(3T_1)) = (\eta_{\underline a}(t)^{3/2}, H)$.  We recall that we must stop at $H$ because below $H$, the conditions of \Cref{lem:movement} are not met.

Hence, taking $\Lambda_3 = \Lambda_2/2$, we may argue as before to obtain the lower bound
\[
	n(3T_1,x,\theta) \geq C_{\Lambda_1} \delta_1^{\Lambda_1,H} e^{- \frac{3\gamma_\infty T}{10}},
\]
for all $(x,\theta) \in \mathcal{E}_{3T_1,\Lambda_3/2}^{X_3,\Theta_3}$.  To be clear, we note that 
\[
	 \mathcal{E}_{3T_1,\Lambda_3/2}^{X_3,\Theta_3}
		= \Big\{(x,\theta) \in \R\times\Theta: \frac{(x - \eta_{\underline a}(T)^{3/2})^2}{\eta_{\underline a}(T)} + (\theta - H)^2 \leq \frac{\Lambda_1}{16}
		\Big\}.
\]

\smallskip

{\bf \# Step 3bis: Moving to near $\underline \theta$ and growing larger}

\smallskip

At this point, we know that $n$ is bounded below by a small value far to the right at time $3T_1$.  Since $T_1 \ll T$ when $T$ is large, we have a large amount of time left for $n$ to grow to order one.

To begin, we obtain estimates on the growth by using the principal eigenvalue of the cane toads operator with trade-off on a large enough domain.  For $r,s>0$, consider the following spectral problem in both variables $(\xi,\theta)$:
\begin{equation*}
\begin{cases}
\theta \varphi_{\xi\xi}   +  \varphi_{\theta\theta}  + (1-m) \varphi   = \gamma_{r,s} \varphi \,, \qquad \text{ on } (\eta_{\underline a}(T)-r,\eta_{\underline a}(T)+r) \times (\underline\theta, \underline\theta + s),\\
\varphi_\theta (\cdot,\underline\theta) = 0\,, \quad \varphi (\cdot,\underline\theta + s) = 0\,, \quad \varphi(\eta_{\underline a}(T)\pm r,\cdot) = 0.
\end{cases}
\end{equation*}
The eigenvector (up to a multiplicative constant) $\varphi$ is given by 
\begin{equation*}
\varphi(\xi,\theta)= \cos\left( \frac{\pi}{2}\frac{\xi-\eta_{\underline a}(T)}{r}\right) V_{r,s}(\theta),
\end{equation*}
where we have introduced the eigenelements $(\gamma_{r,s},V_{r,s})$ with normalization $\|V_{r,s}\|_\infty = 1$ by
\begin{equation*}
\begin{cases}
V_{r,s}''  + \left( - \frac{\pi^2 \theta}{4r^2} - \gamma_{r,s} + (1-m) \right) V_{r,s}   = 0 \,, \quad \theta \in (\underline\theta, \underline\theta + s),\\
V_{r,s}'(\underline\theta) = 0, \quad V_{r,s}(\underline\theta + s) = 0, \quad V_{r,s} > 0.
\end{cases}
\end{equation*}
One can pass to the limit $\lim_{s\to\infty}\lim_{r\to\infty} \gamma_{r,s} = \gamma_\infty$ since $\lim_{s\to\infty}\lim_{r\to\infty} V_{r,s} = Q$.  Thus we may fix $r_0$ and $s_0$ sufficiently large so that $\gamma_{r_0,s_0} > 8\gamma_\infty/10$.

Define $U_{r_0,s_0} = [\eta_{\underline a}(T)-r_0,\eta_{\underline a}(T)+r_0] \times [\underline\theta, \underline\theta + s_0]$.  We now apply the Harnack inequality on $[3T_1, 3T_1 + 1] \times U_{r_0,s_0}$ to obtain
\begin{equation}\label{eq:n_lower_bound_harnack}
	C e^{- \frac{3\gamma_\infty T}{10}}
		\leq \sup_{U_{r_0,s_0}} n(3T_1, \cdot, \cdot)
		\leq C_{r_0,s_0} \inf_{U_{r_0,s_0}} n(3T_1 + 1, \cdot, \cdot),
\end{equation}
where we combined $C_{\Lambda_1} \delta_1^{\Lambda_1,H}$ into one constant $C$ above.

We wish to use $\varphi$ and the bound above to create a sub-solution of $n$.
To deal with the nonlocal term, we use the Harnack inequality given by \Cref{lem:harnack}, to obtain,
\[
	\rho(t,x) \leq C_{s_0}' n(t,x,\theta) + \frac{\gamma_\infty}{10},
\]
for all $(x,\theta) \in U_{r_0,s_0}$.  Hence we have that
\begin{equation}\label{eq:n_diff_inequality}
n_t \geq \theta n_{xx} + n_{\theta\theta} + n \left(1 - \frac{\gamma_\infty}{10} - m - C_{r_0,s_0}' n \right).
\end{equation}
Now define
\[
	w(t,x,\theta) = \min\left(\frac{\gamma_\infty}{10C_{r_0,s_0}'},\frac{C}{C_{r_0,s_0} +1} \right) e^{-\frac{3\gamma_\infty T}{10}} e^{\lambda'(t - 3T_1 - 1)} \varphi(x,\theta),
\]
on $[3T_1 + 1, T]\times U_{r_0,s_0}$, where we set $\lambda' = (3\gamma_\infty T)/(10(T - 3T_1 - 1)).$ We point out that $\lambda' - \gamma_{r_0,s_0} + \gamma_\infty/10 \leq -\gamma_\infty/10$ for $T$ sufficiently large, since then $\lambda' \leq 4\gamma_\infty/10$ and $\gamma_{r_0,s_0} \geq 8\gamma_\infty/10$.  In addition, by the choice of $\lambda'$, we have that $C_{r_0,s_0}' w \leq \gamma_\infty/10$ for all $t \in [3T_1+1,T]$.  Our goal is to compare $w$ and $n$.  This is possible since $w$ satisfies the equation
\begin{equation}\label{eq:w2}
	w_t - \theta w_{xx} - w_{\theta\theta} - w\left(1 - \frac{\gamma_\infty}{10}- m\right)
		= \left(\lambda' - \gamma_{r_0,s_0} + \frac{\gamma_\infty}{10} \right)w \leq - \frac{\gamma_\infty}{10} w  \leq -C_{r_0,s_0}' w^2.
\end{equation}
Thus \eqref{eq:w2} and~\eqref{eq:n_diff_inequality} implies that $n$ is a super-solution to $w$.

In addition, due to~\eqref{eq:n_lower_bound_harnack}, we have that $w(3T_1+1,\cdot) \leq n(3T_1+1,\cdot)$ on $U_{r_0,s_0}$.  Hence, the maximum principle implies that $w \leq n$ on $[3T_1+1,T]\times U_{r_0,s_0}$.  

Evaluating this at time $t = T$ and position $(x,\theta) = (\eta_{\underline a}(T)^{3/2},\underline\theta)$ yields
\[
		{\frac{\gamma_\infty}{10C_{r_0,s_0} C_{r_0,s_0}'}V_{r,s}(\underline\theta)}
		=\frac{\gamma_\infty}{10C_{r_0,s_0} C_{r_0,s_0}'} \varphi(\eta_{\underline a}(T)^{3/2}, \underline\theta)
		= w(T, \eta_{\underline a}(T)^{3/2}, \underline\theta) \leq n(T,{\eta_{\underline a}(T)^{3/2}},\underline\theta).
\]
Since the left hand side does not depend on $T$ and {$V_{r,s}$ is positive due to the maximum principle}, this provides the desired uniform lower bound.  For any $x \leq \eta_{\underline a}(T)^{3/2}$, we may simply decrease $c_2$ appropriately and argue as above, thus finishing the proof of Proposition~\ref{prop:lowerboundacc}.
\end{proof}

\section{Proof of \Cref{prop:tw}: construction of travelling waves}\label{sec:tw}

We now construct a travelling wave solution for~\eqref{eq:main} when $m$ is not sub-linear.  We use these travelling waves later to obtain the spreading results of \Cref{thm:cauchy_finite}.  We recall that a function $n$ solving~\eqref{eq:main} is a \textit{travelling wave} solution with speed $c^* \in \R^+$ if it can be written $n(t,x,\theta)= \mu \left( \xi:= x - c^* t , \theta \right)$
where $\mu$ satisfies
\begin{equation}\label{eqkinwave}
\begin{cases}
- c^* \mu_{\xi} = \theta  \mu_{\xi \xi} + \mu_{\theta  \theta} + \mu  (1 - m - \nu),& \qquad \text{on } \R\times \Theta,\\
\mu_\theta (\cdot,\underline\theta) = 0.&
\end{cases}
\end{equation} 
where $\nu$ is the macroscopic density associated to $\mu$; that is $\nu = \int_\Theta \mu \left( \cdot, \theta \right) d \theta$.

We follow the standard strategy. First, we explain how to compute the minimal speed of propagation of possible solutions. Then, we solve an approximated problem in a bounded slab in both directions $\xi$ and $\theta$. The resolution of the slab problem is mainly a combination of \cite{AlfaroCovilleRaoul,BouinCalvez}. Finally, we let the slab tend to $\R \times \Theta$ to obtain a wave that we prove to have the minimal speed. 

We split the section into two parts.  First, we cover the case where $m(\theta)/\theta \to \infty$ as $\theta\to\infty$.  Second, we discuss the modifications necessary in the case where $m(\theta)/\theta$ tends to a positive constant.

\subsection{Case one: $\lim_{\theta \to + \infty} m(\theta)/\theta = +\infty$.}\label{sec:caseone}

\subsubsection{Spectral problems and minimal speeds}

We start the construction by defining speeds of propagation for the travelling waves. Recall that for any $\lambda >0$, the spectral problem defining $c_\lambda$ is written as follows
\begin{equation*}
		Q_\lambda'' + \left[  \lambda^2 \theta - \lambda c_\lambda + (1 -  m(\theta)) \right] Q_\lambda  = 0, \qquad \theta \in \Theta,
\end{equation*}
accompanied by the Neumann condition $Q_\lambda'(\underline \theta) = 0$.  Since $\lim_{\theta \to + \infty} m(\theta)/\theta = +\infty$, this spectral problem may be solved. We then define $\lambda^* = \min\argmin\{c_\lambda: \lambda > 0\}$ and let $c^* = c_{\lambda^*}$ and $Q^* = Q_{\lambda^*}$.  We fix the normalization that $\|Q_\lambda\|_\infty = 1$.

For $\tau \in [0,1]$, define the spatial diffusivity function $d^\tau$ by $d^\tau(\theta) = \underline \theta  + \tau \left( \theta - \underline \theta \right).$
In general, we suppress the dependence on $\tau$ in this section except in \Cref{sec:homotopy} where the $\tau$ variable is the main concern.  Since we first construct a travelling wave solution on a bounded slab, for any spatial decay rate $\lambda > 0$ and any $b>0$, we also introduce the following approximate spectral problems:
\begin{equation}\label{eq:spectral_problem}
\begin{cases}
 Q_{\lambda,b} ''+ \left( - \lambda c_{\lambda,b} + d \lambda^2   + \left( 1 - m \right) \right) Q_{\lambda,b} = 0, \qquad \text{ on } (\underline\theta, \underline\theta + b),\\
Q_{\lambda,b}' \left( \underline\theta \right) = 0,\qquad 
Q_{\lambda,b}  \left( \underline\theta + b\right) = 0, \qquad
Q_{\lambda,b} > 0.
\end{cases}
\end{equation}
%
By the Krein-Rutman theorem, there exists a positive solution $(c_{\lambda,b}, Q_{\lambda,b})$ to this problem. 
%
We now describe the function $\lambda \mapsto c_{\lambda,b}$.  By looking at the limit as $b$ tends to infinity, the principle eigenvalue of the operator $\partial_\theta^2 + (1-m)$ tends to $\gamma_\infty$.  Hence, when $b$ is sufficiently large, since $d^\tau \lambda^2 > 0$, we have $\lambda c_\lambda \geq \gamma_\infty/2 > 0$, so that $c_\lambda$ is uniformly bounded away from zero. Moreover, we  have that $\lambda c_\lambda \to \gamma_\infty$ as $\lambda$ tends to zero.  Using the Rayleigh quotient, it is easy to check that $\lambda c_\lambda = \mathcal{O}_{\lambda \to + \infty}\left(\lambda^2\right)$.  As a result, since $\lambda \mapsto c_\lambda$ is continuous, it admits a global minimum on $(0,\infty)$.  This minimum defines the minimal speed $c_{b}^{\tau*} := \min_{\lambda > 0} c_{\lambda,b}$. The smallest minimizer is denoted $\lambda_b^*$.

We also introduce the family of eigenvectors $Q_b^\delta$, that appears naturally as boundary conditions at the back for the slab problem below. For any $\delta \geq 0$ and $b\in (0,\infty]$, let us define $Q_b^\delta$ by solving the following eigenvalue problem  
\begin{equation}\label{eq:eigenpb}
\begin{cases} 
	\left( Q_b^\delta \right) '' + \left( 1 - (1-\delta) m \right)Q_b^\delta = \gamma_b^\delta Q_b^\delta, \qquad \text{ on }\theta\in (\underline\theta, \underline\theta + b),\\
	\left( Q_b^\delta \right)' \left( \underline\theta \right) = 0,
	\qquad Q_b^\delta \left(\underline\theta + b \right) = 0,
	\qquad \int_{(\underline\theta, \underline\theta + b)} Q_b^\delta(\theta)\, d\theta = \gamma_b^\delta,
	\qquad Q_b^\delta > 0.	
\end{cases}
\end{equation}
We denote by $Q_b$ the particular case of $Q_b^{\delta = 0}$. The positivity of the eigenvalue $\gamma_b^\delta$ is ensured when $b$ is sufficiently large as $\gamma_b^\delta \geq \gamma_b^0$. 
%

\subsubsection{The problem in a slab}\label{sec:slab}

Given $a,b, \delta, \underline\epsilon >0$ and $\tau \in [0,1]$, we define the problems 
on the slab $(-a,a) \times (\underline\theta, \underline\theta + b)$ as follows:
\begin{equation}\label{eq:slab}
\begin{cases}
		-c_{a,b}^\tau \big( \mu_{a,b}^\tau \big)_\xi  - d^\tau \big( \mu_{a,b}^\tau \big)_{\xi \xi} -  \big( \mu_{a,b}^\tau \big)_{\theta\theta} =\big[ \mu_{a,b}^\tau \big]_+ \big(1 - m - \nu_{a,b}^\tau \big),&\text{on } (-a,a) \times (\underline\theta, \underline\theta + b), \\
		\big( \mu_{a,b}^\tau \big)_{\theta}   (\cdot, \underline \theta ) = 0,
		\qquad \mu_{a,b}^\tau(\cdot, \underline \theta + b ) = 0,\\
		\mu_{a,b}^\tau(-a,\cdot) = Q_b  , \qquad \mu^a(a,\cdot)  = 0
		\qquad \mu_{a,b}^\tau(0,\underline\theta) = \underline\eps,
\end{cases}
\end{equation}
where $\nu_{a,b}^\tau := \int_{(\underline\theta, \underline\theta + b)} \mu_{a,b}^\tau(\cdot,\theta) d\theta$ and where we use the ``positive'' part notation $x_+ := x \1_{x\geq 0}$.  We are, momentarily, suppressing the dependence of $\mu$ on $\underline\epsilon$ and $\delta$.

We note that we seek {\em positive} travelling waves.  This is a consequence of the maximum principle. 
Indeed, suppose that $\mu_{a,b}^\tau$ attains a negative minimum at some point $(\xi_0,\theta_0) \in [-a,a] \times [\underline\theta, \underline\theta + b]$. Necessarily, $\xi_0 \neq \pm a$ due to the Dirichlet boundary conditions and $\theta_0 \neq \underline \theta, \underline\theta+b$ due to the Neumann and Dirichlet boundary conditions, respectively. Hence, $(\xi_0,\theta_0) \in (-a,a) \times (\underline\theta, \underline\theta + b)$, so, by continuity, we can find an open set $\mathcal{V} \subset (-a,a) \times (\underline\theta, \underline\theta + b)$ containing $(\xi_0,\theta_0)$ such that we  have,  
\begin{equation*}
-c_{a,b}^\tau \left( \mu_{a,b}^\tau \right)_\xi  - d^\tau \left( \mu_{a,b}^\tau \right)_{\xi \xi} -  \left( \mu_{a,b}^\tau \right)_{\theta\theta} = 0, \qquad \text{on } \mathcal{V}.
\end{equation*}
By the maximum principle, this would imply that $\mu$ is a negative constant, which is impossible.

We now comment on the problem \eqref{eq:slab}. The unknowns are the speed $c_{a,b}^\tau$ and the profile $\mu_{a,b}^\tau$. Without the supplementary renormalization condition $\mu_{a,b}^\tau(0,\underline\theta) = \underline\epsilon$, the problem is underdetermined. Indeed, this additional condition is needed to ensure compactness of the family $(c_{a,b}^\tau,\mu_{a,b}^\tau)$ when $a$ tends to $\infty$, since the limit problem is translationally invariant. The boundary condition in $-a$ is chosen this way since we heuristically expect that the distribution of the population converges towards $Q$ at the back of an invasion front.  Although we fix this boundary condition in the slab, let us recall again that in general the behavior at the back of the wave for the limit problem is not easy to determine due to possible Turing instabilities.   Lastly, we note that, due to the nonlocal term, the equation for $\mu$,~\eqref{eq:slab}, does not admit a comparison principle.



We follow the standard construction -- obtaining a priori bounds on the slab, using Leray-Schauder degree theory to obtain a solution on the slab, and then carefully taking the limit as the slab approximates $\R\times [\underline\theta,\infty)$.
For notational convenience, we  omit the subscript $a,b$ and the superscript $\tau$ in $\mu_{a,b}^\tau$ and $c_{a,b}^\tau$.

\subsubsection*{An upper bound on $c$}

We first obtain a general upper bound on $c$, which comes from the linearized problem.

\begin{lemma}\label{lem:upboundc}
For any $\underline\epsilon > 0$, there exists $a_0(\underline\epsilon,b)$ such that if $a \geq a_0(\underline\eps,b)$, then any solution $(c,\mu)$ of the slab problem \eqref{eq:slab} satisfies $c \leq c_b^*$.
\end{lemma}

\begin{proof}

We follow the classical approach: we find a super-solution for a related problem. Since $\mu \geq 0$, $\mu$ is a sub-solution to the linearized problem.  In other words,
\begin{equation}\label{eq:n}
-c \mu_{\xi}  \leq d \mu_{\xi\xi} + \mu_{\theta\theta}  + \left( 1 - m \right)\mu, \qquad \text{on } (-a,a) \times (\underline\theta, \underline\theta + b).
\end{equation}

Let us assume by contradiction that $c > c_b^*$, then the family of functions $\psi_A ( \xi, \theta ):= A e^{- \lambda_b^*\xi} Q_b^* ( \theta )$ is indeed a family of super-solutions to the linear problem:
\begin{align}\label{eq:psi_A}
- c \left(\psi_A\right)_{\xi} = c \lambda_b^*\psi_A > \lambda_b^* c_b^* \psi_A &= \theta  \left(\psi_A\right)_{\xi\xi}  + \left(\psi_A\right)_{\theta\theta}  + (1-m )\psi_A \\& \geq d  \left(\psi_A\right)_{\xi\xi}  + \left(\psi_A\right)_{\theta\theta}  + (1-m )\psi_A, \qquad \text{on } (-a,a) \times (\underline\theta, \underline\theta + b) \nonumber.
\end{align}
Since $Q_b^*$ is positive and $\mu$ is bounded and since both functions are $C^2$, we have $\mu \leq \psi_A$ for $A$ sufficiently large.  Hence, define
\begin{equation*}
A_0 = \inf \left\lbrace A \; \vert \; \psi_A \geq \mu \text{ on } [-a,a]\times[\underline\theta, \underline\theta+ b] \right\rbrace. 
\end{equation*}
Necessarily, $A_0 > 0$ and there exists a point $(\xi_0, \theta_0)$ in $[-a,a]\times[\underline\theta, \underline\theta+ b]$ where $\psi_{A_0}$ touches $\mu$: $\mu(\xi_0 , \theta_0) = \psi_{A_0}(\xi_0 , \theta_0).$  This point minimizes $\psi_{A_0} - \mu $ but {\em cannot} be in $(-a,a) \times (\underline\theta, \underline\theta + b)$ if the normalization $\underline\eps$ is well chosen.  Indeed, combining~\eqref{eq:n} and~\eqref{eq:psi_A}, we have
\begin{equation*}
- c \left(\psi_{A_0} - \mu \right)_{\xi} - d \left( \psi_{A_0} - \mu \right)_{\xi\xi} -  \left( \psi_{A_0} - \mu \right)_{\theta\theta} - (1-m) \left( \psi_{A_0} - \mu \right) > 0, \text{ on } (-a,a) \times (\underline\theta, \underline\theta + b).
\end{equation*}
But, if $(\xi_0,\theta_0)$ is in the interior, this inequality cannot hold since at $(\xi_0,\theta_0)$ we  have
\begin{equation*}
	\psi_{A_0} - \mu = 0,
	\quad  \left(\psi_{A_0} - \mu \right)_\xi = 0,
	\quad \text{ and }
	\quad d  \left( \psi_{A_0} - \mu \right)_{\xi\xi} +\left( \psi_{A_0} - \mu \right)_{\theta\theta} \geq 0.
\end{equation*}
Next we rule out the boundaries. First, $\xi_0 \neq a$ since $\psi_{A_0}(a,\theta_0) > 0$.  
 Moreover, $(\xi_0,\theta_0)$ cannot lie where both $\psi_{A_0}$ and $\mu$ satisfy Neumann boundary conditions thanks to the Hopf maximum principle.  Next, we exclude the left boundary $\{\xi=-a\}$ due to the normalization. Indeed, if $\xi_0 = -a$, then $\psi_{A_0} (\xi_0, \theta_0) = Q_b(\theta_0)$ and thus $A_0 = e^{- \lambda_b^* a} Q_b(\theta_0)/Q_b^*(\theta_0)$. Using the definition of $A_0$, we have 
\begin{equation*}
	\underline\epsilon = \mu(0,\underline\theta) < \frac{Q_b(\theta_0)}{Q_b^*(\theta_0)}e^{- \lambda_b^* a} Q_b^* (\underline \theta ).
\end{equation*}
We thus define $a_0(\underline\epsilon,b)$ to be sufficiently large that this inequality cannot hold when $ a > a_0(\underline\epsilon,b)$.

From above, it follows that $\theta_0 = b$ and $\xi_0 \neq \pm a$.  By our choice of $A_0$, we may assume without loss of generality that $(\psi_{A_0} - \mu)_\xi(\xi_0,b) = 0$ since otherwise we could lower $A_0$ further, contradicting its definition.  However, the Hopf maximum principle implies that $(\psi_{A_0} - \mu)_{\xi}(\xi_0,b) < 0$.  This is a contradiction, finishing the proof.

\end{proof}

\subsubsection*{A uniform bound on the trait tails, for $c \in [0,c^* +1]$.}

Since the trait space is unbounded, we must prove the following lemma about the tails of $n$ in $\theta$.

\begin{lemma}[The tails of $\mu$]\label{lem:tails}
Assume $c \in \left[ 0 , c^* + 1 \right]$, $\tau \in [0,1]$, $\delta > 0$, and recall that $Q_\infty^\delta$ is defined in \eqref{eq:eigenpb}. Then there exists a constant $C_{\text{tail}}$ depending on $\delta$ such that on $[-a,a]\times[\underline\theta, \underline\theta+ b]$ we have
\begin{equation*}
	\mu
		\leq C_{tail} \left(1 + \|\mu\|_{L^\infty([-a,a]\times[\underline\theta,\underline\theta+b])}\right) Q_\infty^\delta.
\end{equation*}
\end{lemma}

\begin{remark}
We point out that it is in this proof that we see the importance of modifying the trade-off function with the parameter $\delta$.  Without this, a uniform bound in $b$ would be impossible.
\end{remark}

\begin{proof}
Our strategy is to find a relevant super-solution using spectral problems.
To do this, define, for $\theta_\delta$ to be chosen later, the function 
\begin{equation*}
	\psi  := \max \left( 1 , \frac{\Vert \mu \Vert_{L^{\infty} ( [-a,a] \times [\underline\theta, \underline\theta + b] )}}{\min_{[\underline\theta,\theta_\delta]} Q_\infty^\delta}, \left\|\frac{Q_b}{Q_\infty^\delta}\right\|_{L^\infty} \right) \, Q_\infty^\delta,
	\quad\text{ on }  (-a,a) \times (\underline\theta, \underline\theta + b).
\end{equation*}
 We have $\mu \leq \psi$ on $[-a,a] \times [\underline\theta,\theta_\delta]$ by construction. It satisfies
\begin{equation*}
- c \psi_{\xi}  - d \psi_{\xi\xi} - \psi_{\theta\theta} - \left( 1 - m \right)\psi  = \big( - \gamma_\infty^\delta + \delta m \big) \psi, \quad \text{ on } (-a,a) \times (\underline\theta, \underline\theta + b).
\end{equation*}
Define $\theta_\delta := m^{-1}\left(\delta^{-1}  \gamma_\infty^\delta \right)$. This definition is possible by \Cref{hyp:m}. 
The function $\psi$ is then a super-solution of the linearized problem on $(-a,a) \times (\theta_\delta,\underline\theta+b)$. 
On $(-a,a) \times \left\lbrace \underline\theta + b \right\rbrace$ and $\left\lbrace a \right\rbrace \times (\theta_\delta,\underline\theta+b)$, we see that $\mu \leq \psi$ since $\mu$ satisfies Dirichlet boundary conditions and $\psi$ is positive there. The only boundary that remains to be checked is $\left\lbrace -a \right\rbrace \times (\theta_\delta,\underline\theta+b)$. There, it is true by construction.  

We point out that, using similar reasoning, there exists $\theta_\delta'$, independent of $b$, such that $Q_b$ is a sub-solution of $Q_\infty^\delta$ on $[\theta_\delta',b]$. Further, using elliptic regularity, $Q_b$ is uniformly bounded above independent of $b$, and, by the maximum principle $Q_\infty^\delta$ is uniformly bounded below on $[\underline\theta, \theta_\delta']$.  Taken together these  facts imply that we may bound $\|Q_b/Q_\infty^\delta\|_\infty$ by a constant that is independent of $b$.

\end{proof}

\subsubsection*{A uniform bound over the steady states, for $c \in \left[ 0 , c^* + 1 \right]$.}

We now come to the crucial part of the procedure, that is, deriving uniform bounds on steady states for \eqref{eq:slab}.  We assume \textit{a priori} that the speed is nonnegative, since we prove later that the problem \eqref{eq:slab} does not admit any solution with $c=0$. We extend a similar argument used in \cite{BouinCalvez} 

\begin{lemma}[A priori estimates for bounded $c$]\label{lem:nc}
Assume $c \in \left[ 0 , c^* + 1 \right]$, $\tau \in [0,1]$, $a\geq 1$ and $b$ is sufficiently large.
 Then there exists a constant $C_0$, depending only on $\underline\theta$ and $m$, such that any solution $(c,\mu)$ of \eqref{eq:slab} satisfies
\begin{equation*}
\Vert \mu \Vert_{L^\infty\left( [-a,a]\times[\underline\theta,\underline\theta+b] \right) }\leq C_0\,.
\end{equation*}
\end{lemma}

\begin{proof}
We define $M_{a,b} = \|\mu\|_{L^\infty([-a,a]\times[\underline\theta, \underline\theta + b])}$ and want to prove that $M_{a,b}$ is in fact bounded uniformly in $a$ and $b$.  We can assume that $M_{a,b} \geq 1$ because, if not, we are finished.  As $Q_\infty^\delta$ tends to zero as $\theta$ tends to infinity, choose $\overline \theta$ such that, if $\theta \geq \overline\theta$, $Q_\infty^\delta(\theta) < 1/(4C_{tail})$, where $C_{tail}$ is as in \Cref{lem:tails}.  As a consequence of this, along with the bound $M_{a,b} \geq 1$, we have, for all $\theta \geq \overline \theta$,
\[
	\mu(\cdot, \theta) \leq C_{tail}(1 + M_{a,b}) Q_\infty^\delta(\theta)
		\leq 2C_{tail} M_{a,b} Q_\infty^\delta(\theta)
		\leq M_{a,b}/2.
\]
Hence, the maximum of $\mu$ can only be attained by points $(x_{\max},\theta_{\max})$ with $\theta_{\max} < \overline \theta$.  Then we may interpret ``sufficiently large'' in the assumptions to mean $b \geq \overline \theta + 1$.  In addition, we assume that $\theta_{\max} + 1 > \underline\theta$ in order to avoid complications due to the boundary; however, such complications may be easily dealt with by a simple reflection procedure outlined in~\cite{BerestyckiMouhotRaoul,Turanova}.  Finally, we assume that $x_{\max} \neq \pm a$, since the maxima here are controlled independent of $a$ and $b$. 

Fix any $p \in (1,\infty)$, and we use local elliptic regularity results to obtain a H\"older bound on $\mu$.  Indeed, by, e.g.~\cite[Theorem~9.13]{GilbargTrudinger},
\begin{align*}
	\|\mu&\|_{W^{2}_p\left( B_1(x_{\max},\theta_{\max}) \right)}\\
		&\leq C_p\left( \|\mu\|_{L^p\left( B_2(x_{\max},\theta_{\max}) \right)} + \|\mu(1 - m -\nu)\|_{L^p\left( B_2(x_{\max},\theta_{\max}) \right)} + \|Q_b\|_{W^2_p\left( B_2(x_{\max},\theta_{\max}) \right)}\right)\\
		&\leq C_p\left(\|\mu\|_{L^p\left( B_2(x_{\max},\theta_{\max}) \right)} + \|\mu m\|_{L^p\left( B_2(x_{\max},\theta_{\max}) \right)} + \|\mu \nu\|_{L^p\left( B_2(x_{\max},\theta_{\max}) \right)} + 1\right)\\
		&\leq C_p\left( \|\mu\|_{L^\infty\left( B_2(x_{\max},\theta_{\max}) \right)} + \|\mu m\|_{L^\infty\left( B_2(x_{\max},\theta_{\max}) \right)} + \|\mu \nu\|_{L^\infty\left( B_2(x_{\max},\theta_{\max}) \right)} + 1\right).
\end{align*}
Above, the constant $C_p$ changes line-by-line but is uniform in $a$ and $b$.  Indeed, the operator in~\eqref{eq:slab} is bounded and coercive independently of $b$ on $B_2(x_{\max},\theta_{\max})$ since $\theta_{\max} \leq \overline \theta$.  Notice that
\[
	\|m\|_{L^\infty(B_2(x_{\max},\theta_{\max}))} = \|m\|_{L^\infty(\theta_{\max}-2,\theta_{\max}+2)} \leq \|m\|_{L^\infty(\underline \theta, \overline \theta + 2)} \leq C,
\]
since $\theta_{\max} \leq \overline \theta$.  In addition, \Cref{lem:tails} implies that $\|\nu\|_{L^\infty(-a,a)}\leq C(1 + M_{a,b})$.  These two bounds, along with the elliptic regularity bound above, imply that
\[
	\|\mu\|_{W^2_p(B_1(x_{\max}, \theta_{\max}))}
		\leq C(1 + M_{a,b} + M_{a,b}^2)
		\leq 3C M_{a,b}^2.
\]
where $C$ is a constant independent of $a$ and $b$ that will change line-by-line in the sequel.
Choosing $p$ large enough, we obtain via Sobolev embedding that for any fixed $\alpha \in (0,1)$,
\begin{equation}\label{eq:lenya}
	[\mu]_{C^{1+\alpha}(B_1(x_{\max}, \theta_{\max}))} \leq C\|\mu\|_{W^2_p(B_1(x_{\max}, \theta_{\max}))}
		\leq C M_{a,b}^2.
\end{equation}
Applying the Gagliardo-Nirenberg interpolation inequality for the function $\theta \mapsto \mu \left( x_{\max}, \theta \right)$ yields
\begin{equation*}
	\| \mu \left( x_{\max}, \cdot \right) \|_{L^\infty_\theta\left((\theta_{\max}-1,\theta_{\max}+1)\right)}
		\leq C \| \mu \left( x_{\max}, \cdot \right) \|_{L^1_\theta\left((\theta_{\max}-1,\theta_{\max}+1)\right)}^{\frac{1+\alpha}{2+\alpha}}
			\left[ \mu \left( x_{\max}, \cdot \right) \right]_{C^{1+\alpha}_\theta\left((\theta_{\max}-1,\theta_{\max}+1)\right)}^{\frac{1}{2+\alpha}}.
\end{equation*}
Using~\eqref{eq:lenya}, $\| \mu \left(x_{\max}, \cdot \right) \|_{L^\infty_\theta\left((\theta_{\max}-1,\theta_{\max}+1)\right)} = M_{a,b}$, and $\| \mu \left( x_{\max}, \cdot \right) \|_{L^1_\theta\left((\theta_{\max}-1,\theta_{\max}+1\right)} \leq \nu(x_{\max})$, we obtain that
\begin{equation}\label{eq:lenya1}
	M_{a,b}
		\leq C \nu(x_{\max})^{\frac{1+\alpha}{2+\alpha}}
			M_{a,b}^{\frac{2}{2+\alpha}}.
\end{equation}
Since $(x_{\max}, \theta_{\max})$ is the location of a maximum, then at this point we have that
\begin{equation}\label{eq:maximum_principal}
	0 \leq -c\mu_x - d^\tau \mu_{xx} - \mu_{\theta\theta}
		= \mu(1 - m - \nu).
\end{equation}
Thus, $\nu(x_{\max}) \leq 1$.  Using this along with~\eqref{eq:lenya1}, we obtain $M_{a,b} \leq C M_{a,b}^{\frac{2}{2+\alpha}}$.
This gives a bound on $M_{a,b}$ since $2/(2+\alpha) < 1$.
\end{proof}

\subsubsection*{Non-existence of solutions of the slab problem when $c=0$.}


\begin{lemma}[Lower bound for $\mu(0,\underline\theta)$ when $c=0$]\label{lem:bottom}
There exists $\underline\epsilon_0 > 0$ such that if $a$ and $b$ are large enough and $\tau\in[0,1]$, then any nonnegative solution of the slab problem $(0,\mu)$ satisfies $\mu(0,\underline\theta) > \underline\epsilon_0$. 
\end{lemma}

\begin{proof}
The idea of the proof is to build a sub-solution.  Since the full problem does not enjoy the comparison principle, we use the Harnack inequality to compare with a local equation. 

For $r,s>0$, consider the following spectral problem in both variables $(\xi,\theta)$:
\begin{equation}\label{eq:evpb}
\begin{cases}
d \varphi_{\xi\xi}   +  \varphi_{\theta\theta}  + (1-m) \varphi   = \gamma_{r,s} \varphi \,, \qquad \text{ on } (-r,r) \times (\underline\theta, \underline\theta + s),\\
\varphi_\theta (\cdot,\underline\theta) = 0\,, \quad \varphi (\cdot,\underline\theta + s) = 0\,, \quad \varphi(\pm r,\cdot) = 0.
\end{cases}
\end{equation}
%
The eigenvector (up to a multiplicative constant) $\varphi$ is given by 
\begin{equation*}
\forall (\xi,\theta) \in (-r,r) \times (\underline\theta, \underline\theta + s), \qquad \varphi(\xi,\theta)= \cos\Big( \frac{\pi}{2}\frac{\xi}{r}\Big) V_{r,s}(\theta),
\end{equation*}
where we have introduced the eigenelements $(\gamma_{r,s},V_{r,s})$ with normalization $\|V_{r,s}\|_\infty = 1$ by
\begin{equation*}
\begin{cases}
V_{r,s}''  + \big( - \frac{\pi^2}{4r^2} d - \gamma_{r,s} + (1-m) \big) V_{r,s}   = 0 \,, \quad \theta \in (\underline\theta, \underline\theta + s),\\
V_{r,s}'(\underline\theta) = 0, \quad V_{r,s}(\underline\theta + s) = 0, \quad V_{r,s} > 0.
\end{cases}
\end{equation*}
Since $\lim_{r\to\infty} V_{r,s} = Q_s$, then  $\lim_{s\to\infty}\lim_{r\to\infty} \gamma_{r,s} = \gamma_\infty$.
Thus we may fix 
$s$ and $r$ sufficiently large so that $\gamma_{r,s} > \gamma_\infty/2$.  Notice that we require $r \leq a$ and $s \leq b$ for $\varphi$ to be a sub-solution.

In order to compare $\varphi$ and $\mu$, we must first estimate $\nu$. We decompose it as follows 
\begin{equation*}
\nu = \int_{(\underline\theta, \underline\theta + s)} \mu(\cdot,\theta) d\theta + \int_{(\underline\theta + s, \underline\theta + b)} \mu(\cdot,\theta) d\theta, 
\end{equation*} 
and then estimate the two terms separately. 
First, due to the Harnack inequality 
there exists a constant $C_{r,s}$, depending only on $r$ and $s$, such that
\begin{equation*}
C_{r,s} \mu(0,\underline\theta)
	\geq \sup_{(-r,r) \times (\underline\theta, \underline\theta + s)} \mu(\xi , \theta).
\end{equation*}
Moreover, due to \Cref{lem:tails}, 
\begin{equation*}
\int_{(\underline\theta + s, \underline\theta + b)} \mu(\cdot,\theta) d\theta \leq C_{tail} (1 + \|\mu\|_\infty)  \int_{(\underline\theta + s, +\infty)} Q_\infty^\delta(\theta) d\theta \leq \frac{\gamma_\infty}{4},
\end{equation*}
when $s$ is chosen sufficiently large. To compare \eqref{eq:slab} to \eqref{eq:evpb} on $(-r,r) \times (\underline\theta, \underline\theta + s)$, we write
\begin{equation*}
d \mu_{\xi\xi} + \mu_{\theta\theta} + (1-m) \mu = \mu \nu \leq \mu \left( C_{r,s} s \mu(0,\underline\theta) + \gamma_\infty/4 \right).
\end{equation*}
We deduce from this computation that if $\mu(0,\underline\theta)  \leq \gamma_\infty/(4C_{r,s} s)$, we have \begin{equation*}
\left( C_{r,s}s  \mu(0,\underline\theta) + \gamma_\infty/4  \right) \mu(\xi,\theta) < (\gamma_\infty/2) \mu(\xi,\theta) < \gamma_{r,s} \mu(\xi,\theta), \qquad \text{on } (-r,r) \times (\underline\theta, \underline\theta + s).
\end{equation*}
Hence, if $\mu(0,\underline\theta) \leq \gamma_\infty/(4C_{r,s}s)$, then $\mu$ is a super-solution of \eqref{eq:evpb}.  If $\mu(0,\underline\theta) \leq \gamma_\infty/(4C_{r,s}s)$ does not hold, we are finished.  As such, we assume that it does not for the sake of contradiction.

We now use the same arguments as in the proof of \Cref{lem:upboundc}. Indeed, we define
\begin{equation*}
\alpha_0 = \max \left\lbrace \alpha \in \R^+ \, : \, \alpha \varphi < \mu \text{ on } [-r,r] \times [\underline\theta, \underline\theta + s]  \right\rbrace,
\end{equation*}
so that $u:= \mu - \alpha_0 \varphi$ has a zero minimum at some $(\xi_0,\theta_0) \in [-r,r] \times [\underline\theta, \underline\theta + s]$ and satisfies 
\begin{equation*}
\begin{cases}
d u_{\xi\xi} + u_{\theta\theta} + (1-m) u  < \gamma_{r,s} u,& \quad \text{on } (-r,r) \times (\underline\theta, \underline\theta + s),\\
u_\theta (\cdot,\underline\theta) = 0, \quad u(\cdot,\underline\theta + s) > 0, \quad u(-r,\cdot) > 0, \quad u(r,\cdot)  > 0. \\
\end{cases}
\end{equation*}
As in Lemma \ref{lem:upboundc} this cannot hold. This contradiction implies that $\mu(0,\underline\theta) > \frac{\gamma_\infty}{4 s C_{r,s}}$.

\end{proof}

\subsubsection{The homotopy argument}\label{sec:homotopy}

\begin{proposition}[Existence of a solution in the slab]\label{slabsol}
There are positive constants $\underline\eps_0$ and $b_0$ such that if $\underline \eps < \underline\eps_0$ and $b > b_0$ then there exists $a_0(b,\underline\eps)$ such that if $a > a_0(b,\underline\eps)$ then the slab problem~\eqref{eq:slab} with the normalization condition $\mu_0(0,\underline \theta) = \underline\epsilon$ has a solution $(c,\mu)$.
\end{proposition}

\begin{proof}
Fix $\beta \in (0,1)$.  Given $\mu \in \mathcal{C}^{1,\beta}((-a,a) \times (\underline\theta, \underline\theta + b))$, we consider the problem

\begin{equation}\label{eq:tauslab}
\begin{cases}
-c Z_\xi^\tau - d^\tau Z_{\xi\xi}^\tau - Z_{\theta\theta}^\tau =  \mu_+ (1 - m -  \nu_\mu),\qquad \text{on }  (-a,a) \times (\underline\theta, \underline\theta + b), \\
Z_\theta^\tau(\cdot,\underline \theta) = 0, \quad Z^\tau(\cdot, \underline \theta + b) = 0, \quad
Z^\tau(-a,\cdot) = Q_b, \quad Z^\tau(a,\cdot)  = 0.
\end{cases}
\end{equation}

Here we have introduced the notation $\nu_\mu$ to emphasize that it corresponds to the density associated to $\mu$ (\textit{i.e.} $\nu_\mu(\cdot) := \int_{(\underline\theta,\underline\theta + s)} \mu(\cdot,\theta) d\theta$) and not to $Z^\tau$.  We introduce the map 
\begin{equation*}
\mathcal{K}_{\tau}: (c,\mu) \to \left(\underline \epsilon - \mu(0,\underline \theta) + c , Z^{\tau} \right), 
\end{equation*}
where $Z_{\tau}$ is the solution of the previous linear system \eqref{eq:tauslab}. 
Since~\eqref{eq:tauslab} is elliptic the map $\mathcal{K}_{\tau}: X \to X$ is a compact map where $X := \mathbb{R} \times \mathcal{C}^{1,{\beta}} \left( (-a,a) \times (\underline\theta, \underline\theta + b)\right)$ with the norm $\Vert (c , \mu) \Vert := \max \left( \vert c \vert, \Vert \mu \Vert_{ \mathcal{C}^{1,\beta}} \right)$. Moreover, it depends continuously on the parameter $\tau \in \left[ 0 , 1 \right]$. Solving the problem 
\eqref{eq:slab} is equivalent to proving that the kernel of $\text{Id} - \mathcal{K}_1$ is non-trivial. We can now apply the Leray-Schauder theory. 

Fix $C_0$ from \Cref{lem:nc}, $\underline\epsilon < \underline\epsilon_0$, and $a > a_0(\underline\epsilon,b)$.  We define the open set
\begin{equation*}
\mathcal{B} = \left \lbrace (c,\mu) \; \vert \; 0 < c < c^* + 1, \; \Vert \mu \Vert_{\mathcal{C}^{1,\beta}\left( (-a,a) \times (\underline\theta, \underline\theta + b)\right)} < C_0 + 1 \right\rbrace.
\end{equation*} 
The a priori estimates of Lemmas \ref{lem:upboundc} and \ref{lem:nc}, give that for all $\tau \in \left[ 0 , 1\right]$ and sufficiently large $a$, the operator $ \text{Id}  - \mathcal{K}_{\tau}$ cannot vanish on the boundary of $\mathcal{B}$. 
The Leray-Schauder degree theory yields the homotopy invariance
\begin{equation*}
	\text{deg}\left( \text{Id} - \mathcal{K}_{1} , \mathcal{B} , 0 \right) 
		= \text{deg}\left(  \text{Id} - \mathcal{K}_{0} , \mathcal{B} , 0 \right).
\end{equation*}
To compute $\text{deg}\left(  \text{Id} - \mathcal{K}_{0} , \mathcal{B} , 0 \right)$, we investigate the problem corresponding to $\tau = 0$.
This problem is equivalent to solving the following symmetrized equation on $(-a,a)\times(-b,b)$:
\begin{equation*}
\begin{cases}
	-c \widetilde{Z}_\xi ^0 - \underline\theta \widetilde{Z}_{\xi\xi}^0  - \widetilde{Z}_{\theta\theta}^0 = \mu_+ (1 - m(|\theta-\underline\theta| + \underline\theta) - \frac12 \nu_\mu), \qquad \text{on }(-a,a)\times(-b,b), \\
	\widetilde{Z}^0(\cdot, - b) = \widetilde{Z}_\theta^0(\cdot,b) = 0,
	\quad \widetilde{Z}^0(-a,\cdot) = Q_b(\vert \cdot \vert + \underline\theta),
	\quad \widetilde{Z}^0(a,\cdot)  = 0.
\end{cases}
\end{equation*}
Using this formulation, that $-1 
= \deg\left(  \Id - \mathcal{K}_{0} , \mathcal{B} , 0 \right)$ is exactly the purpose of \cite[Proposition 3.9]{AlfaroCovilleRaoul}. 
Since $ -1 = \deg(\Id - \mathcal{K}_0, \mathcal{B},0)$, then $\mathcal{K}_1$ has a fixed point.  This finishes the proof.
\end{proof}

\subsubsection{Construction of a spatial travelling wave in $\R\times (\underline\theta,\underline\theta + b)$}

We now use the solution of the slab problem \eqref{eq:slab} given by Proposition~\ref{slabsol} to construct a travelling wave solution. For this purpose, we first pass to the limit $a \to \infty$ to obtain a profile in $\R\times(\underline\theta,\underline\theta+b)$. Then we prove that this profile has speed $c_b^*$ and the correct asymptotics as $\xi \to \pm\infty$.

\begin{lemma}\label{lem:convslab}
Let $\underline\eps < \underline\eps_0$. There exists $c \in \left[ 0 , c_b^* \right]$ such that the system
\begin{equation}\label{convslab2}
\begin{cases}
- c_0 \mu_{\xi} - \theta \mu_{\xi\xi} - \mu_{\theta\theta} = \mu (1 - m - \nu),  & \text{on } \R\times(\underline\theta,\underline\theta+b), \\
 \mu_\theta(\cdot,\underline\theta) = 0, \qquad \mu(\cdot, \underline\theta + b) = 0, \\
\end{cases}
\end{equation}
has a nonnegative solution $\mu \in \mathcal{C}_b^2\left( \R \times (\underline\theta,\underline\theta + b)\right)$ satisfying $\mu(0,\underline\theta) = \underline\eps$. 
\end{lemma}
\begin{proof}
Using the uniform bounds above, along with classical elliptic regularity theory, and applying the Arzela-Ascoli theorem, we take the limit $a\to \infty$ to obtain the lemma.
\end{proof}

\subsubsection*{The profile is travelling with the minimal speed $c_b^*$.}

In this section, we present the arguments showing that the constructed front has the minimal speed $c_b^*$. It roughly goes as follows.  First, any solution with $c \leq c^*_b$ that is bounded away from zero at $\{\theta = \underline\theta\}$ is, in fact, bounded from below on $\{\theta=\underline\theta\}$ by a constant $\omega>0$ that is independent of $\mu$.  Contrasting this lower bound with the choice of normalization, we argue that $\lim_{\xi\to \infty} \mu(\xi,\underline\theta) = 0$.  Finally, by building oscillating sub-solutions, we show that this limit holds only when $c \geq c^*_b$, which, in view of \Cref{lem:upboundc} implies that $c = c^*_b$. The ideas used here are similar to those in \cite{AlfaroCovilleRaoul,BouinCalvez}.


\begin{lemma}\label{lem:inf}
There exists $\omega > 0$ such that any solution $(c,\mu)$ constructed in \Cref{slabsol}
with $c \in \left[ 0 , c_b^*\right]$ and $\inf_{\xi \in \R} \mu(\xi,\underline\theta)  > 0$ satisfies $\inf_{\xi \in \R} \mu(\xi,\underline\theta) > \omega$.
\end{lemma}

\begin{proof}
Fix $\varphi$ as in \Cref{lem:bottom} and fix any point $\xi_0$.  Let $\phi(\xi,\cdot) = \varphi(\xi - \xi_0, \cdot)$.  We construct a sub-solution of $\mu$ using $\phi$ to obtain a bound on $\mu$ at $\xi_0$ independent of $\xi_0$.  Recall that we fix $r$ and $s\leq b$ sufficiently large that $\gamma_{r,s} \geq \gamma_\infty/2$.

By the Harnack inequality, there exists $C_s$ sufficiently large that $\mu(\xi,\theta') \leq C_s \mu(\xi,\theta)$ for any $\theta, \theta'\in(\underline \theta, \underline\theta+s)$ and any $\xi$.  When $s$ is large enough, we may estimate $\nu$ as above to obtain,
\[
	\forall (\xi,\theta) \in \R\times (\underline\theta, \underline\theta + s), \qquad \nu(\xi) \leq C_s \mu(\xi,\theta) + \frac{\gamma_\infty}{4}.
\]

From the $L^\infty$ estimate on $\mu$ and the construction of $\varphi$, there exists $\overline\alpha$ such that $\overline\alpha \phi(\xi_0,\underline\theta) > \mu(\xi_0,\underline\theta)$. 
On the other hand, by the Harnack inequality, we have $\underline\alpha \phi \leq \mu$ where $\underline\alpha := \frac{ \inf_\R \mu(\cdot,\underline\theta)}{ C_s} > 0$ where $V_{r,s}$ is defined in \Cref{lem:bottom} and $\|V_{r,s}\|_\infty = 1$. As a consequence, we can define 
\begin{equation*}
\alpha_0:= \sup \lbrace \alpha > 0 \, : \, \alpha \phi \leq \mu  \text{ on } [-r+\xi_0,r+\xi_0] \times [\underline\theta, \underline\theta + s] \rbrace.
\end{equation*}
As usual, 
there exists $(\xi_{\max} , \theta_{\max})$ such that $\mu - \alpha_0 \phi$ has a minimum of zero at this point.  It is easy to check that this point must occur in the interior of $(-r+\xi_0,r+\xi_0)\times(\underline\theta,\underline\theta+s)$.  Hence,
\begin{equation}\label{eq:phi}
\begin{split}
0 & \geq  - \theta_{\max}( \mu - \alpha_0 \phi)_{\xi\xi} -    ( \mu - \alpha_0 \phi )_{\theta \theta} - c ( \mu - \alpha_0 \phi )_\xi, \\
& \geq  \big(1 - m - C_s \mu - \frac{\gamma_\infty}{4} \big)\mu + \alpha_0 ( \theta_{\max} \phi_{\xi\xi} + \phi_{\theta \theta} + c \phi_\xi ), \\
& \geq  \big(1 - m - C_s \mu - \frac{\gamma_\infty}{4} \big)\mu - \alpha_0 ( (1-m) \phi - \gamma_{r,s} \phi ) - \alpha_0 \frac{c\pi}{2r} \sin\Big( \frac{\pi}{2} \frac{\xi_{\max}-\xi_0}{r} \Big) V_{r,s}(\theta_{\max}),  \\
& \geq  \big( \gamma_{r,s} -\frac{\gamma_\infty}{4}- C_s \mu  \big) \mu - \alpha_0 \frac{c\pi}{2r} \sin\Big( \frac{\pi}{2} \frac{\xi_{\max} - \xi_0}{r}\Big) V_{r,s}(\theta_{\max}).
\end{split}
\end{equation}
From the Harnack inequality, we deduce that $C_s \mu(\xi_{\max},\theta_{\max}) \geq \mu(\xi_{\max},\underline\theta) $.  Recalling also the inequalities $\inf_{\R} \mu(\cdot,\underline\theta) > 0$, $c \leq c_b^*$, $\alpha_0 \leq \overline\alpha$, $\gamma_{r,s} \geq \gamma_\infty/2$, we get
\begin{equation*}
\mu(\xi_{\max} , \theta_{\max}) \geq \displaystyle  \frac{ \gamma_\infty}{4C_s} - \frac{ \pi \overline \alpha c_b^*}{2 r s \mu(\xi_0,\underline\theta)}.
\end{equation*}
Taking $r$ sufficiently large we have that $\mu(\xi_{\max}, \theta_{\max}) \geq \gamma_\infty/(8C_s)$.  
Since $\mu$ and $\alpha_0 \phi$ coincide at $(\xi_{\max},\theta_{\max})$, we have $\alpha_0 \geq \frac{\gamma_\infty }{8C_s}$. We are finished by noting that
\begin{equation*}
	\mu(\xi_0,\underline\theta) \geq \alpha_0 \phi(\xi_0,\underline\theta)
		\geq \frac{\gamma_\infty}{8 C_s } V_{r,s}(\underline\theta).
\end{equation*}
\end{proof}

Next we show that the front has the required limits at infinity.  The second part of this proposition is crucial in the sequel in showing that our front moves with speed $c \geq c^*_b$.

\begin{proposition}\label{prop:limits}
Any solution $(c,\mu)$ of~\eqref{convslab2}
with $c \in[0,c^*_b]$, and $\mu(0,\underline\theta) = \underline\eps$ satisfies\smallskip
\begin{enumerate}
	\item[(i)]\label{point:1} For all sufficiently large $s < b$, there exists $\alpha_s >0$ such that $\mu > \alpha_s Q_s$ on $(-\infty,0)\times(\underline\theta, \underline \theta + s)$;
%
	\item[(ii)]\label{point:2} $\lim_{\xi \to +\infty} \mu(\xi,\cdot) =0.$
\end{enumerate}
\end{proposition}
\begin{proof}
We start with the proof of (i).  Recall from the proof of Lemma \ref{lem:inf} that $\mu$ satisfies 
\begin{equation}\label{eq:mu_local_eqn}
\forall \left( \xi , \theta \right) \in \R \times (\underline\theta,\underline\theta + s), \qquad - c \mu_{\xi}  - \theta \mu_{\xi\xi} - \mu_{\theta\theta} \geq \left(1 - m - C_s \mu - \frac{\gamma_\infty}{4}  \right)\mu.
\end{equation}

Let $\underline r$ be sufficiently large so that, for $s$ sufficiently large, $\gamma_{\underline r,s}> \gamma_\infty/2$. For any $r\geq \underline r$,
 define $\varphi_r = \alpha \cos\left(\frac{\pi\xi}{2r}\right)V_{r,s}(\theta)$, as above,
with $\alpha = \min \left( \frac{\underline\eps}{2 \tilde C_s}, \frac{\gamma_\infty}{8 C_s s}  \right)$, where $\tilde C_s$ is defined below.  

We first show that $\varphi_{\underline r}\leq \mu$ on $[- \underline r , 0] \times \left[\underline\theta,\underline\theta + s\right]$. We have $\varphi_{\underline r}  = 0 < \mu$ on $\{-\underline r\} \times \left[\underline\theta,\underline\theta + s\right]$. 
The Harnack inequality, applied on $[-\underline r,0] \times \left[\underline\theta,\underline\theta + s\right]$, yields a postive constant $\widetilde C_s$ such that
\begin{equation}\label{eq:HHn}
\widetilde{C}_s \inf_{[-\underline r,0] \times \left[\underline\theta,\underline\theta + s\right]} \mu \geq  \mu(0,\underline\theta) = \underline\eps.
\end{equation}
Recall that $\|V_{r,s}\|_\infty = 1$.  Thus, on $[-\underline r , 0] \times \left[\underline\theta,\underline\theta + s\right]$, using \eqref{eq:HHn}, we have that $\varphi_{\underline r} \leq \mu$.


As a consequence we can define 
\begin{equation*}
r_0:= \sup \lbrace r \geq \underline r \, : \, \varphi_r \leq \mu \text{ on }\left(- r , 0\right] \times \left[\underline\theta,\underline\theta + s\right]  \rbrace.
\end{equation*}
We now prove that $r_0 = \infty$ by contradiction. Suppose that $r_0 < \infty$. Then, as usual,
there exists $(\xi_0 , \theta_0) \in (-\infty,0] \times \left[\underline\theta,\underline\theta + s\right]$ such that $\mu - \varphi_{r_0}$ has a zero minimum at this point. Note that $\theta_0\neq \underline\theta + s$ and $\xi_0 \neq r_0$ since $\varphi_r$ vanishes at those points. 
Moreover, $\xi_0$ cannot be $0$ since $\varphi_r(0,\theta) < \mu(0,\theta)$ on $(\underline\theta,\underline\theta+s)$, by our choice of $\alpha$ and by~\eqref{eq:HHn}.  
Thus, $(\xi_0 , \theta_0) \in (-r_0,0)\times(\underline\theta,\underline\theta+s)$ or $\theta_0 = \underline \theta$.  In either case, the maximum principle and Hopf maximum principle along with~\eqref{eq:phi} and~\eqref{eq:mu_local_eqn} imply that
\begin{equation*}
0
	\geq \left( \frac{\gamma_\infty}{4}- C_s \mu  \right) \mu -  \alpha \frac{c\pi}{2r_0} \sin\left( \frac{\pi \xi_0}{2 r_0}\right) V_{r_0,s}(\theta_0)
	\geq \left( \frac{\gamma_\infty}{4}- C_s \mu  \right) \mu.
\end{equation*}
Above, we used that $s$ and $r$ are large enough that $\gamma_{r,s} > \gamma_\infty/2$.   Thus,
\begin{equation*}
\frac{\gamma_\infty}{4C_s s} \leq \mu(\xi_0,\theta_0) = \varphi_{r_0}(\xi_0,\theta_0) \leq \alpha,
\end{equation*}
which contradicts the definition of $\alpha$.
As a consequence, $r_0 = \infty$.  Since $V_{r,s} \to Q_s$ as $r$ tends to $\infty$, then we have that $\alpha Q_s \leq \mu$, as claimed.

We now prove (ii). By the Harnack inequality and Lemma \ref{lem:tails}, it is sufficient to prove that $\lim_{\xi \to \infty}\mu(\xi,\underline\theta) = 0$. Suppose that there exists $\delta>0$ and a sequence $\xi_n \to + \infty$ such that for all $n \in \N, \;\mu(\xi_n,\underline\theta)\geq \delta$. Adapting the proof of (i), we find $\alpha_{s,\delta}>0$ that for all $n \in \N$,
\begin{equation}\label{eq:last}
\mu(\xi,\underline\theta) \geq \alpha_{s,\delta} Q_s(\underline\theta), \qquad \forall (\xi,\theta) \in \left( -\infty , \xi_n \right] \times [\underline\theta,\underline\theta + s].
\end{equation}
Hence \eqref{eq:last} holds for all $\xi \in \R$.  This contradicts \Cref{lem:inf}, lowering $\underline\eps$ so that $\underline \eps < \omega$ if necessary.
\end{proof}

\begin{proposition}[The front speed is $c_b^*$]\label{prop:minspeed}
%
%
Any nonzero, nonnegative solution $(c,\mu)$ given by \Cref{lem:convslab} satisfying $\inf_{\xi\in\R} \mu(\xi,\underline\theta) = 0$ satisfies $c \geq c_b^*$.
\end{proposition}

\noindent Note that, due to \Cref{lem:upboundc}, Proposition~\ref{prop:limits}, and Proposition~\ref{prop:minspeed}, the solution given by \Cref{lem:convslab} travels with the speed $c=c_b^*$.

\begin{proof}

Assume that $c<c_b^*$.  By analogy with the Fisher-KPP equation, we use oscillating fronts to ``push'' solutions of \eqref{convslab2} up to the speed $c_b^*$. 

Fix $\epsilon_1 > 0$ to be determined.  By choosing $s = b - \epsilon_1/(1+\|\mu\|_\infty)$ and applying the Harnack inequality for any $\theta, \theta' \in [\underline\theta, \underline \theta+s]$, we have that $\mu(\xi,\theta') \leq C_s \mu(\xi,\theta)$.  Note that $C_s$ and $s$ depend only on $b$ and $\epsilon_1$.  Recall that $\|\mu\|_\infty$ is bounded above by \Cref{lem:nc}.  Hence, we have
\[
	\nu(\xi)
		= \int_{\underline\theta}^{\underline\theta + b} \mu(\xi,\theta')d\theta'
		\leq \int_{\underline\theta}^{\underline\theta + s} C_s \mu(\xi,\theta)d\theta' + \int_{\underline\theta+ s}^{\underline\theta + b} \|\mu\|_\infty d\theta'
		= C_s' \mu(\xi,\theta) + \epsilon_1,
\]
where $C_s' = s C_s$.
As a result, $\mu$ satisfies
\begin{equation*}
\forall \left( \xi , \theta \right) \in \R \times (\underline\theta,\underline\theta +s), \qquad - c \mu_{\xi}  - \theta \mu_{\xi\xi} - \mu_{\theta\theta} \geq \left(1 - m - C_s' \mu - \eps_1 \right)\mu.
\end{equation*}

We explain below how to construct a compactly supported sub-solution using a relevant spectral problem in the complex plane, see also \cite{BouinCalvez} for a related argument. 

Recall from~\eqref{eq:spectral_problem} that for any $\lambda, \eps \in \R^+$, $Q_{\lambda,s}$ solves the spectral problem
\begin{equation}\label{eq:spectral_problem_s}
\begin{cases} 
	Q_{\lambda,s}''+ \left( - \lambda c_{\lambda,s,\epsilon} + \theta \lambda^2   + \left( 1 - \epsilon_1 - \epsilon -  m \right) \right) Q_{\lambda,s} = 0\,, \qquad \theta \in (\underline\theta,\underline\theta + s), \\
	Q_{\lambda,s}'\left( \underline\theta \right) = 0, \quad Q_{\lambda,s}\left( \underline\theta + s\right) = 0, \quad Q_{\lambda,s} > 0,
\end{cases}
\end{equation}
with $c_{\lambda, s, \epsilon} = c_{\lambda, s} - (\epsilon_1 + \epsilon)/\lambda$. 
Let $c^*_{s,\epsilon}$ be the minimum, occurring at $\lambda_{s,\epsilon}^*$, of $c_{\lambda,s,\epsilon}$ over all $\lambda \in \R^+$. From the explicit expression of  $c_{\lambda,s,\epsilon}$, we obtain $c^*_{s,\epsilon} < c^*_s <c^*_b$. By fixing $\eps_1$ sufficiently small, we can ensure $c < c^*_{s,\epsilon=0} < c^*_b$. Thus, there exists $\epsilon_c > 0$ such that
\[
	c = c_{s,\epsilon_c}^* = c_{\lambda_{s,\epsilon_c}^*,s,\epsilon_c} = c_{\lambda_{s,\epsilon_c}^*,s} - \frac{\epsilon_1 +\epsilon_c}{\lambda_{s,\epsilon_c}^*}.
\]




Now consider~\eqref{eq:spectral_problem_s} for complex values of $\lambda$. Perturbation theory, see \cite[Chapter 7, \S 1, \S 2, \S 3]{Kato}, yields that the map $\lambda \mapsto c_{\lambda,s,\epsilon}$ is analytic in $\lambda$ at least in a neighborhood of the real axis.

Our aim is now to  find $\epsilon$ and $\lambda_c := \lambda_{c,R} + i \lambda_{c,I}$ (with $\lambda_{c,I} \neq 0$) such that $c_{\lambda_c,s,\epsilon} = c$.  We note that $\lambda_{c,I}\neq 0$ allows us to construct compactly supported sub-solutions; see below.  We argue using Rouché's theorem (around $\lambda_{s,\epsilon_c}^*$).

Define $f(\lambda) = c_{\lambda, s, \epsilon_c} - c$. From above, we have that $f(\lambda_{s,\epsilon_c}^*) = 0$. Since $\lambda \mapsto Q_{\lambda,s}$ is continuous, 
$\partial_\theta Q_{\lambda_{s,\epsilon_c}^*}(\underline\theta+s) < 0$ due to the Hopf lemma, and $Q_{\lambda_{s,\epsilon_c}^*} > 0$ in $(\underline\theta,\underline\theta+s)$, then $\Real(Q_{\lambda,s}) > 0$ for $\lambda$ sufficiently close to $\lambda_{s,\epsilon_c}^*$. 
Moreover, since the zeros of analytic functions are separated, for any sufficiently small $r \in (0,\lambda_{s,\epsilon_c}^*)$ there is $\delta>0$ such that $|f(\partial B_r(\lambda_{s,\epsilon_c}^*))| \geq \delta > 0$. Thus fix $r \in (0,\lambda_{s,\epsilon_c}^*)$ sufficiently small so that $|f(\partial B_r(\lambda_{s,\epsilon_c}^*))| \geq \delta > 0$ and $\Real(Q_{\lambda,s}) > 0$ for any $\lambda \in B_r(\lambda_{s,\epsilon_c}^*)$. 

Define $g(\lambda) =  c_{\lambda, s, \epsilon}  - c$. Fix $\eps < \eps_c$ close enough to $\eps_c$ such that $0 < \vert \eps_c - \eps\vert/(\lambda_{s,\epsilon_c}^* - r) < \delta$.  Then, on $\partial B_r(\lambda_{s,\epsilon_c}^*)$, we have that
\[
	|f(\lambda) - g(\lambda)|
		= \frac{\vert \eps_c - \eps \vert}{|\lambda|}
		\leq \frac{\vert \eps_c - \eps \vert}{\lambda_{s,\epsilon_c}^* - r}
		< \delta
		\leq |f(\lambda)|
		\leq |f(\lambda)| + |g(\lambda)|.
\]
Hence the hypotheses of Rouch\'e's theorem are met.  Thus, $f$ and $g$ have the same number of zeros in $B_r(\lambda_{s,\epsilon_c}^*)$.  Since $f(\lambda_{s,\epsilon_c}^*) = 0$ then $g$ has at least one zero, $\lambda_c$.  Using the definition of $g$, we have that $c_{\lambda_c,s,\eps} = c$. Moreover, since $\eps < \eps_c$, we have necessarily $c_{s,\eps}^* > c_{s,\eps_c}^* = c$ and thus $\lambda_c \not\in \R$.
%
%
%

We now let
\begin{equation*}
\psi(\xi,\theta):= \text{Re} \big( e^{- \lambda_c \xi} Q_{\lambda_c,s} \left( \theta \right)\big) = e^{-\lambda_{R} \xi}\left[ \text{Re} \left(Q_{\lambda_c,s}(\theta)\right) \cos (\lambda_{I} \xi) + \text{Im} \left(Q_{\lambda_c,s}(\theta)\right) \sin (\lambda_{I} \xi) \right].
\end{equation*} 
Notice that $\psi(0,\theta) > 0 > \psi(\pm \lambda_{c,I}^{-1} \pi, \theta)$ for all $\theta \in [\underline\theta, \underline\theta + s)$.
  Hence, by the continuity of $\psi$, there exists an open subdomain $\mathcal{D} \subset [-\lambda_{c,I}^{-1}\pi,\lambda_{c,I}^{-1}\pi]\times [\underline\theta, \underline\theta+s]$ such that $\psi>0$ on $\mathcal{D}$ and vanishes on $\partial \mathcal{D}$,
except possibly where $\mathcal{D}$ intersects $\lbrace \theta = \underline\theta \rbrace$ where $\psi$ satisfies Neumann boundary conditions.

By construction of $\psi$, we have
\begin{equation*}
- c \psi_\xi - \theta \psi_{\xi\xi} - \psi_{\theta \theta} - (1 - \eps_1 - m) \psi = - \eps \psi, \qquad \text{on } \mathcal{D}.
\end{equation*}
Thus, for all $\alpha \geq 0$, the function $v:= \mu - \alpha \psi$ satisfies
\begin{equation}\label{eq:chris}
- c v_\xi - \theta v_{\xi\xi} - v_{\theta \theta} - (1 - \eps_1 -m)v \geq \alpha \eps\psi - \left( C_s' \mu\right) \mu = \left( \eps - C_s' \mu\right) \mu - \eps v.
\end{equation}

Arguing as \Cref{lem:inf}, there exists $\alpha_0$ such that $v$ attains a zero minimum at $(\xi_0,\theta_0) \in \overline{\mathcal{D}}$. The minimum point is in the interior due to the boundary conditions. From \eqref{eq:chris} evaluated at $(\xi_0,\theta_0)$, we deduce $\mu(\xi_0,\theta_0) \geq \frac{\eps}{C_s'}$. Applying the Harnack inequality on $\mathcal{D}$, we conclude that $\mu(0,\underline\theta) \geq \frac{\eps}{C_s'}$ after possibly changing the constant $C_s'$. 

We emphasize that the renormalization $\mu(0,\underline\theta) = \underline\eps$, which is the only reason for which \eqref{eq:slab} is not invariant by translation in $\xi$, is not used here. Hence, we note that our argument did not depend on the spatial variable $\xi$.  As such, we can conclude that $\mu(\xi,\underline\theta) \geq \frac{\eps}{C_s'}$ for all $\xi$.  
We then obtain $\inf_{\xi \in \R} \mu(\xi,\underline\theta) \geq \frac{\eps}{C_s'}$. This contradicts the property $\inf_{\xi \in \R} \mu(\xi,\underline\theta)= 0$.
\end{proof}

\subsubsection*{Taking the limit $b\to\infty$}

Since we have uniform bounds on $\mu_b$, we may take locally uniform limits $\mu_b \to \mu_\infty$ as $b$ tends to infinity.  Since the speed associated with $\mu_b$ is $c_b^*$ and since $c_b^* \to c^*$ then $\mu_\infty$ satisfies equation~\eqref{eqkinwave}.  We need only check that the limit is non-trivial.  Proposition~\ref{prop:limits}.(i) gives that $\liminf_{\xi\to-\infty} \nu_\infty(\xi) > 0$.  On the other hand, we may argue exactly as in Proposition~\ref{prop:limits}.(ii) in order to show that $\limsup_{\xi\to\infty} \mu_\infty(\xi,\underline\theta) = 0$.  Hence $\mu_\infty$ is our traveling wave with speed $c^*$, finishing the proof of Proposition~\ref{prop:tw} in the case when $m$ is super-linear.

\subsection{Case two (the critical case): $\lim_{\theta\to\infty} m(\theta)/\theta = \kappa^2 > 0$}

In this section, we show the differences appearing in the critical case. We now assume that the trade-off function takes the following form
\begin{equation*}
m(\theta) = \kappa^2 \theta + \widetilde m(\theta), 
\end{equation*}
where $\kappa > 0$ and $\widetilde m(\theta)/\theta \to 0$.

We start by constructing the speeds of propagation for the travelling waves. This is where the assumption on $m$ plays the main role. For any spatial decay rate $\lambda > 0$, we re-introduce the spectral problem on $[\underline\theta, \underline\theta + b]$ for $b$ sufficiently large and possibly infinite:
\begin{equation*}
\begin{cases} 
Q_{\lambda,b}''+ \left( - \lambda c_{\lambda,b} + \left( \lambda^2 - \kappa^2 \right) \theta  + 1 - \widetilde m(\theta)  \right) Q_{\lambda,b} = 0\,, \qquad \theta \in (\underline\theta, \underline\theta + b), \\
 Q_{\lambda,b} ' \left( \underline\theta \right)= Q_{\lambda,b}(\underline\theta + b) = 0, \quad Q_{\lambda,b} > 0.
\end{cases}
\end{equation*}

We may observe that as long as either $\lambda < \kappa$ or $b < \infty$, the previous spectral problem has a unique solution $c_{\lambda,b}$ as in the previous sub-section. Moreover, one can prove again that $\lim_{\lambda \to 0} \lambda c_{\lambda,b} = \gamma_b^{\delta=0} \leq \lambda c_{\lambda,b}$.  However, the unique issue of this case is that when $\lambda > \kappa$ and $b = \infty$ this spectral problem does not have any solution since, denoting $c_\lambda = c_{\lambda,\infty}$,
\begin{equation*}
\lim_{\theta \to + \infty} \left( - \lambda c_\lambda + \theta \lambda^2   + \left( 1 - m(\theta) \right) \right) = + \infty.
\end{equation*}
We point out that, the borderline case $\lambda = \kappa$ need not have a solution but it does if $\widetilde m(\theta) \to \infty$.

Hence, the function $\lambda \mapsto c_\lambda$ has an infimum on $(0,\kappa)$ but may not have a minimum. We define the minimal speeds $c^*_b := \inf_{\lambda\in\R^+} c_{\lambda,b}$, if $b < \infty$, and $c^* := \inf_{\lambda \in (0,\kappa)} c_\lambda$ otherwise. Note that
\begin{equation*}
c_{b}^* = \inf_{\lambda \in \R^+} c_{\lambda,b} \leq \inf_{\lambda \in (0,\mu)} c_{\lambda,b} < \inf_{\lambda \in (0,\mu)} c_{\lambda} = c^*. 
\end{equation*}
Moreover $\lim_{b \to +\infty} c_{b}^* = c^*$, since $\lim_{b \to +\infty} c_{\lambda,b} = + \infty$ when $\lambda > \kappa$. As a consequence, the entire proof of case one can be reproduced in this case to prove \Cref{prop:tw}.

\section{Proof of \Cref{thm:cauchy_finite}: linear spreading for the Cauchy problem}\label{sec:cauchy_finite}

We first prove the following lower bound on the propagation speed of any initial data.

\begin{lem}\label{lem:finite_speed_lower_bd}
	Under the assumptions of \Cref{thm:cauchy_finite}, there exists $\underline n>0$ such that if $c < c^*$ then 
	\[
		\liminf_{t\to\infty} \inf_{|x| < ct} n(t,x,\underline\theta) \geq \underline n.
	\]
\end{lem}
\begin{proof}
Fix $c < c^*$, fix large constants $a,b>0$, and fix $\alpha< 1$ to be determined.  Let $\mu$ be the solution to the slab problem on $(-a,a)\times(\underline\theta,\underline\theta+b)$ given in \Cref{sec:slab} solving
\[
	- c_{a,b} \mu_x - \theta \mu_{\xi\xi} - \mu_{\theta\theta} = \mu(\alpha - m - \nu),
\]
with the boundary conditions used above.  We point out that while \Cref{sec:slab} only proves the existence of $\mu$ for $\alpha = 1$, the general case may obtained similarly.  From the work above, it follows that we may choose $a$ and $b$ sufficiently large and $1-\alpha$ sufficiently small that its speed $c_{a,b} \in (c,c^*]$.  Set $\underline n (t,x,\theta) = A^{-1}\mu(x - c_{a,b}t , \theta)$, where $A$ is a positive constant to be determined.

With this definition, using the elliptic Harnack inequality for $\mu$, we have that $\underline n$ satisfies
\[
	\underline n_t \leq \theta \underline n_{xx} + \underline n_{\theta\theta} + \underline n (\alpha - m - A C_b^{-1} \underline n).
\]
On the other hand, arguing as in the proof of \Cref{thm:acceleration}, we use \Cref{lem:harnack} along with the decay of $n$ to obtain that, for $t\geq 1$,
\[
	n_t \geq \theta n_{xx} + n_{\theta\theta} + n(1 -\epsilon_b - m - C_b n),
\]
where $\epsilon_b$ is a parameter which tends to zero as $b$ tends to $\infty$.  We take $b$ sufficiently large so that $1-\epsilon_b \geq \alpha$.  Choosing $A \geq C_b^2$, we see that $\underline n$ is a sub-solution to $n$ for all $t \geq 1$.  Since $n(1,x,\theta) > 0$ for all $(x,\theta)$, we may choose $A$ sufficiently large that $\underline n(0,x,\theta) \leq n(1,x,\theta)$.  As a result, the maximum principle implies that $\underline n(t,x,\theta) \leq n(t+1,x,\theta)$ for all $t\geq 0$ and all $(x,\theta)$.  Recalling the definition of $\underline n$, we have that, for all $(x,\theta) \in (-a,a)\times(\underline\theta,\underline\theta+b)$,
\[
	A^{-1} \mu(x, \theta) \leq n(t+1, x + c_{a,b}t, \theta).
\]

Arguing as in Step \#3bis in \Cref{thm:acceleration}, we may ``wait'' to remove the dependence on $n(1,\cdot,\cdot)$.  Namely, we may find a constant $\mu_0>0$, independent of $n(1,\cdot,\cdot)$, such that $\mu_0 \leq n(t + t_0, x + c_{a,b} t, \theta)$ for all $(x,\theta) \in (-1,1)\times (\underline\theta, \underline \theta+1)$and $t$ sufficiently large. Evaluating at $\theta = \underline\theta$ finishes the claim.\end{proof}

In addition, we obtain a matching upper bound.

\begin{lem}\label{lem:finite_speed_upper_bd}
	Under the assumptions of \Cref{thm:cauchy_finite}, if $c>c^*$ then
	\[
		\lim_{t\to\infty} \sup_{x > ct,~\theta\in\Theta} n(t,x,\theta) = 0.
	\]
\end{lem}
\begin{proof}
	Fix $c > c^*$ and $c_0 \in (c^*,c)$.  Then there exists $\lambda_0>0$ such that $c_{\lambda_0} = c_0$.  Let $Q_{\lambda_0}$ solve~\eqref{eq:specQlambda} with $c_0$ and $\lambda_0$ as above.  Let $\overline n(t,x,\theta) = C_0 e^{\lambda_0 C_0 - \lambda_0(x-c_0 t)	} Q_{\lambda_0}(\theta)$, where $C_0$ is the constant in~\eqref{eq:n_0}.  Then, by construction, $n_0 \leq \overline n(0,\cdot,\cdot)$ and $\overline n$ satisfies
	\[
		\overline n_t = \theta \overline n_{xx} + \overline n_{\theta\theta} + \overline n(1-m).
	\]
	The maximum principle implies that $n \leq \overline n$. Thus,
	\[
		\lim_{t\to\infty}\sup_{x\geq ct, \theta\in\Theta} n(t,x,\theta)
			\leq \lim_{t\to\infty}\sup_{x\geq ct, \theta\in\Theta} n(t,x,\theta)
			= \lim_{t\to\infty} C_0 e^{\lambda_0 C_0 - \lambda_0(c-c_0)t} \|Q_{\lambda_0}\|_\infty
			= 0,
	\]
	as desired.  This concludes the proof.

\end{proof}

The combination of Lemma~\ref{lem:finite_speed_lower_bd} and Lemma~\ref{lem:finite_speed_upper_bd} yields Theorem~\ref{thm:cauchy_finite}.

\section{Proof of Proposition~\ref{prop:extinction}: Extinction of $n$ when $\gamma_\infty \leq 0$}

\begin{proof}[{\bf Proof of  Proposition~\ref{prop:extinction}}]
Recall that the eigenvector $Q$ is a natural super-solution for $n$. Indeed, let $\overline n(t,x,\theta) = A Q(\theta)e^{\gamma_\infty t}$, where the constant $A$ is chosen such that $n_0 \leq \overline n(0,\cdot,\cdot)$. Then 
\begin{equation}\label{eq:lintoads}
\begin{cases}
	\overline n_t = \theta \overline n_{xx} + \overline n_{\theta\theta} + \left( 1 - m \right)  \overline n,\\
	\overline n_\theta(\cdot,\underline\theta) = 0.
	\end{cases}
\end{equation}
By the comparison principle, since $n$ is a sub-solution to  \eqref{eq:lintoads}, we have $n \leq \overline n$, and the conclusion of the proposition follows from the negativity of $\gamma_\infty$ when $\gamma_\infty<0$.

When $\gamma_\infty$ is zero, we argue as follows.  Define $v = n/Q$ and notice that $v$ satisfies
\[
	v_t = \theta v_{xx} + \frac{1}{Q^2} (Q^2 v_\theta)_\theta - v\rho.
\]
Multiplying by $Q^2 v$ and integrating the equation above, we have that
\begin{equation}\label{eq:energy}
	\frac{1}{2}\frac{d}{dt}\int Q^2 v^2\, dxd\theta = -\int Q^2\left( \theta |v_x|^2 + |v_\theta|^2\right)\, dx d\theta - \int Q^2 v^2 \rho\, dx d\theta.
\end{equation}
It is enough to show that $v$ tends uniformly to zero for bounded $\theta$.  If not, then there are positive constants $\epsilon$ and $H$ and a sequence of times $t_n\to \infty$, places $x_n$, and traits $\theta_n\in(\underline \theta, H)$ such that
\[
	v(t_n,x_n,\theta_n) \geq 2\epsilon.
\]
Using parabolic regularity along with the uniform bound on $n$, Proposition~\ref{prop:uniform_upper_bound}, we can find a $\gamma$, depending only on $u_0$, $H$, and $m$, such that $v \geq \epsilon$ holds on $[t_n, t_n + \gamma]\times [x_n- \gamma, x_n + \gamma]\times [\theta_n - \gamma, \theta_n + \gamma].$ 
We point out that this clearly implies that $\rho(t,x) \geq 2\gamma \epsilon$ for $(t,x) \in [t_n, t_n + \gamma]\times[x_n - \gamma, x_n + \gamma]$.  The combination of these two facts gives us that
\[
	\int_{t_n}^{t_n+\gamma}\int Q^2 v^2 \rho\, dx d\theta ds
		\geq C_{H,\gamma,m} \gamma^2 \epsilon^3.
\]
Since this inequality is true for all $t\in[t_n,t_n + \gamma]$ for every $n$, we may conclude that
\[
	\lim_{t\to\infty} \int_0^t Q^2 v^2 \rho\, dxd\theta ds
		\geq \sum_{n=1}^\infty \int_{t_n}^{t_n+\gamma}\int Q^2 v^2 \rho\, dx d\theta ds = \infty.
\]
%
%
Integrating~\eqref{eq:energy} in time and using the inequality above, we have that
\[
\lim_{t\to\infty}\int Q^2 v^2\, dxd\theta
		\leq \int Q^2 v_0^2\, dxd\theta
		- \lim_{t\to\infty}\int_0^t \int Q^2 v^2 \rho\, dxd\theta ds
		= -\infty.
\]
This is clearly a contradiction since the left hand side is non negative.
\end{proof}

A natural way to obtain a more precise estimate of the rates of decay when $\gamma_\infty = 0$ is to derive a dedicated Nash-type inequality, as in the case for the heat equation.  
This is out of the scope of the present paper, we do not go further in that direction.  
 
\section{A priori bounds for the Cauchy problem and a Harnack inequality}\label{sec:apriori}

\subsection{The uniform bound on $n$}

In this section, we prove a uniform upper bound of $n$ in $L^\infty$.  
Before beginning the technical work, we set some notation that is necessary in the sequel.  First, we define the parabolic cylinder 
\[
	\mathcal{C}_R(z_0) \stackrel{\rm def}{:=} (t_0 - R^2,t_0)\times \left\{ (x,\theta): \vert x - x_0 \vert^2 \leq R^2, \vert \theta - \theta_0 \vert^2 \leq R^2  \right\},
\]
where $z_0 = (t_0,x_0,\theta_0)$ and $R$ is any positive constant.  
In general, we may simply refer to $\mathcal{C}_R(z_0)$ as $\mathcal{C}_R$ when no confusion arises.  
Fixing $\delta > 0$, we recall, on this cylinder, the norms
\[\begin{split}
	& [n]_{\delta/2,\delta, \mathcal{C}_R} = \sup_{(t,x,\theta) \neq (s,y,\eta) \in \mathcal{C}_R} \frac{|n(t,x,\theta) - n(s,y,\eta)|}{(|x-y| + |\theta - \eta| + |t-s|^{1/2})^\delta},~~~~~~\text{ and }\\
	& [n]_{1+\delta/2,2+\delta, \mathcal{C}_R} = [n_t]_{\delta/2,\delta, \mathcal{C}_R} + \sum_{k+\ell = 2} [\partial_x^k\partial_\theta^\ell n]_{\delta/2,\delta,\mathcal{C}_R}.\\
\end{split}\]
We also define the parabolic Sobolev spaces: for any $p \in [1,\infty]$ and $\Omega \subset \R^+ \times \R\times \Theta$, let
\[\begin{split}
	W^{1,2}_p(\Omega) \stackrel{\rm def}{=}
		\left\{ f: \Omega \to \R: \max_{2j + k + \ell \leq 2} \int_{\Omega} |\partial_t^j \partial_x^k \partial_\theta^\ell f|^p dtdxd\theta < \infty \right\}.
\end{split}\]
We endow these with the obvious norm.

Our starting point for obtaining an $L^\infty$ bound on $n$ is the following bound on the tails of the solution, which is very similar to \Cref{lem:tails} for the travelling waves.
\begin{lemma}\label{lem:tailsevol}
Denote, for any $T>0$, $M_T = \sup_{t\in[0,T]} \|n(t,\cdot,\cdot)\|_{L^\infty\left(\R\times\Theta\right)}$. For any $\delta > 0$, there exists $C_\delta$, depending only on $\delta$ such that $n(t,x,\theta) \leq C_\delta M_T Q_\infty^\delta$.  In addition, $\rho(t,x) \leq CM_T$.
\end{lemma}
\begin{proof}

To find a super-solution, define, for $\theta_\delta$ to be chosen later, the function 
\begin{equation*}
	\psi  := \max \left( 1 , \frac{M_T}{\min_{[\underline\theta,\theta_\delta]} Q_\infty^\delta}\right) \, Q_\infty^\delta,
	\quad\text{ on }  \R\times \Theta.
\end{equation*}
We have $n \leq \psi$ on $[0,T] \times \R \times [\underline\theta,\theta_\delta]$ by construction. It satisfies
\begin{equation*}
\psi_t  - \theta \psi_{\xi\xi} - \psi_{\theta\theta} - \left( 1 - m \right)\psi  = \left( - \gamma_\infty^\delta + \delta m \right) \psi, \quad \text{ on } [0,T] \times \R \times \Theta.
\end{equation*}
Define $\theta_\delta$ such that $\theta_\delta \geq m^{-1}\left(\delta^{-1}  \gamma_\infty^\delta \right)$ and $n_0(x,\theta) = 0$ for $\theta \geq \theta_\delta$. This definition is possible by \Cref{hyp:m}. The function $\psi$ is then a super-solution of the linearized problem on $[0,T] \times \R \times [\theta_\delta,+\infty)$. Hence the comparison principle implies that $\psi \geq n$ on $[0,T] \times \R \times \Theta$. This finishes the proof of the first claim.  The second claim follows by simply integrating the inequality in $\theta$.
\end{proof}

With \Cref{lem:tailsevol}, we are now in a position to prove the $L^\infty$ bound on $n$~(Proposition~\ref{prop:uniform_upper_bound}).


\begin{proof}
We recall the notation that $M_T = \sup_{t\in[0,T]} \|n(t,\cdot,\cdot)\|_{L^\infty\left(\R\times\Theta\right)}$ and point out that $M_T$ must be finite since a basic upper bound for the equation is given by the super-solution $e^t M_0$.  Hence, we have that, at worst, $M_T \leq e^T M_0$.  Our goal is to obtain a bound on $M_T$ independent of $T$.

A consequence of \Cref{lem:tailsevol} is that the supremum of $n$ can only be approached by points $(t,x,\theta)$ with $\theta$ sufficiently small.  Hence, by parabolic regularity and translation, we may assume that $M_T$ is achieved at some point $(t_T, x_T, \theta_T)$ with $t_T \in[0,T]$, similarly to~\cite[Section~2]{Turanova} and~\cite[Section~7]{BerestyckiMouhotRaoul}.  In addition, we assume that $\theta_T + 4 > \underline\theta$ in order to avoid complications due to the boundary; however, such complications may be easily dealt with by a simple reflection procedure outlined in~\cite{BerestyckiMouhotRaoul,Turanova}.

We assume without loss of generality that $t_T \geq 5$.  Then 
\begin{equation*}
	0 \leq n_t - \theta_T n_{xx} - n_{\theta\theta}
		= n(1 - m(\theta_T) - \rho) \leq n(1 - \rho).
\end{equation*}
at the point $(t_T,x_T,\theta_T)$, since this is the location of a maximum. Thus, $\rho(t_T,x_T) \leq 1$.

Fix any $p \in (1,\infty)$, and local parabolic regularity results, see e.g.~\cite[Theorem 7.22]{Lieberman}, give
\[
	\|n\|_{W^{1,2}_p\left( \mathcal{C}_1(z_T) \right)}
		\leq C\left( \|n\|_{L^p\left( \mathcal{C}_2(z_T) \right)} + \|n(1 - m -\rho)\|_{L^p\left( \mathcal{C}_2(z_T) \right)}\right)
		\leq C(M_T + M_T^2),
\]
where we used~\Cref{lem:tailsevol} to bound $\rho$.  We point out that the constant $C$, above, depends only on $p$ and $m$.  
With $p$ large enough, we obtain via Sobolev embedding that for any $\delta > 0$,
\[
	[n]_{C^{(1+\delta)/2,1+\delta}(\mathcal{C}_1)} \leq C (M_T + M_T^2).
\]
Applying the Gagliardo-Nirenberg interpolation inequality to $\theta \mapsto n \left( t_T,x_T, \cdot \right)$, we obtain
\begin{equation*}
\| n \left( t_T,x_T, \cdot \right) \|_{L^\infty_\theta\left(\mathcal{C}_1(\theta_T)\right)} \leq C \| n \left( t_T,x_T, \cdot \right) \|_{L^1_\theta\left(\mathcal{C}_1(\theta_T)\right)}^{\frac{1+\delta}{2+\delta}} \left[ n \left( t_T,x_T, \cdot \right) \right]_{C^{1+\delta}_\theta\left(\mathcal{C}_1(\theta_T)\right)}^{\frac{1}{2+\delta}}.
\end{equation*}
Since $\| n \left( t_T,x_T, \cdot \right) \|_{L^\infty_\theta\left(\mathcal{C}_1(\theta_T)\right)} = M_T$ and $\| n \left( t_T,x_T, \cdot \right) \|_{L^1_\theta\left(\mathcal{C}_1(\theta_T)\right)} \leq \rho(t_T,x_T) \leq 1$, we obtain that
\begin{equation*}
M_T \leq C \left( M_T + M_T^2 \right)^{\frac{1}{2+\delta}}.
\end{equation*}
This clearly gives a bound on $M_T$ since $2/(2+\delta) < 1$.  The bound on $\rho$ follows from the combination of this bound and \Cref{lem:tailsevol}.
\end{proof}

\subsection{Comparing $\rho$ and $n$ via a local-in-time Harnack inequality}

With \Cref{lem:tailsevol} and Proposition~\ref{prop:uniform_upper_bound} in hand, we may now state the following Harnack inequality which allows us to compare solutions of the local and nonlocal problems.
\begin{proof}[{\bf Proof of \Cref{lem:harnack}}]
First, using~\Cref{lem:tailsevol} along with the uniform bound from Proposition~\ref{prop:uniform_upper_bound}, we note that $n \leq CMQ_\infty^\delta$.  Fix $R_1 \geq R$ such that $\int_{R_1}^\infty CMQ_\infty^\delta d\theta\leq \epsilon/2$.  Now, by arguing as in~\cite[Theorem~2.6]{AlfaroBerestyckiRaoul}, we may find $C_{R_1,\epsilon,t_0}$ such that, for all $t\geq t_0$,
\[
	\sup_{|x-x_0|, \theta-\underline\theta < R_1} n(t,x,\theta)
		\leq C_{R_1,\epsilon,t_0} \inf_{|x-x_0|, \theta-\underline\theta < R_1} n(t,x,\theta) + \epsilon/(2R_1).
\]
Since the proof in our setting is a straightforward adaptation, we omit it.  Then,
\[\begin{split}
	\rho(t,&x_0) \leq \int_{\underline\theta}^{R_1} \sup_{|x-x_0|, \theta - \underline\theta < R_1} n(t,x,\theta)d\theta + \int_{R_1}^\infty n(t,x,\theta) d\theta\\
		&\leq R_1 C_{R_1,\epsilon} \inf_{|x-x_0|, \theta-\underline\theta < R_1} n(t,x,\theta) + \frac{\epsilon}{2} + \int_{R_1}^\infty CMQ_\infty^\delta d\theta
		\leq R_1 C_{R_1,\epsilon} \inf_{|x-x_0|, \theta-\underline\theta < R} n(t,x,\theta) + \epsilon.
\end{split}\]
This concludes the proof.
\end{proof}

\appendix

\numberwithin{equation}{section}

\section{Appendix: Applying the results of Li-Yau}\label{sec:appendix}

\subsubsection*{Obtaining the bound used in~\Cref{lem:Li_Yau}}

In this section, we briefly describe how to apply the heat kernel bounds of Li-Yau.  To begin, we compute the scalar curvature of $\R\times \Theta$ endowed with the metric $g$ given by~\eqref{eq:metric}.  The scalar curvature, $R$, is defined to be
\[
	R = g^{ij} \left( \partial_k \Gamma_{ij}^k - \partial_j \Gamma_{ik}^j + \Gamma_{ij}^\ell\Gamma_{k\ell}^k - \Gamma_{ik}^\ell\Gamma_{j\ell}^k\right),
\]
where we use Einstein summation notation, i.e.~repeated indices are implicitly summed over.  Here $g^{ij}$ is the $(i,j)$ entry in $g^{-1}$ and $\Gamma_{ab}^c$ is the Christoffel symbol given by the formula
\[
	\Gamma_{ab}^c = \frac{1}{2} g^{ck} \left( \partial_b g_{ka} + \partial_a g_{kb} - \partial_k g_{ab}\right).
\]
It is straightforward to compute that $\Gamma_{11}^1 = \Gamma_{22}^i = \Gamma_{12}^2 = \Gamma_{21}^2 = 0$, that $\Gamma_{12}^1 = \Gamma_{21}^1 = - 1/2\theta$ and that $\Gamma_{11}^2 = 1/2\theta^2$.  Plugging this into the equation for $R$, we see that $R = -2/\theta^2$.  For surfaces, the Ricci and scalar curvatures are equivalent up to a multiplicative factor.  Hence, the above computations bound the Ricci curvature from below.
Hence, $R$ is bounded uniformly below by a constant $-\underline R$, where we set $\underline R = 2/\underline\theta^2$.

We now show how to obtain \Cref{lem:Li_Yau} from the results of \cite{LiYau}.  In our setting, the statement of~\cite[Corollary~3.2]{LiYau}\footnote{Strictly speaking, this result is only valid for a complete Riemannian manifold without boundary.  See below for a discussion of the adaptations to the proof to obtain the result in our setting.} 
is that, for any $\bar a >0$, 
there exists a positive constant $C_0$, depending only on $\|m'\|_\infty$, $\|m''\|_\infty$, $\underline R$, and $\bar a$, such that
\begin{equation}\label{eq:prelim_heat_kernel_bound}
	G(t,x,\theta, 0, \underline\theta)
		\leq \frac{C_0}{\sqrt{\Vol(S_{\bar a t}(t,x,\theta)) \Vol(S_{\bar a t}(t,0,\underline\theta))}} \exp \left\{ C_0t - \frac{4}{5} \zeta(t,x,\theta) \right\}.
\end{equation}
Here, for any $z$, $\eta$, and $s$, $S_{a}(s,z,\eta) := \left\{(z',\eta') \in \R\times \overline{\Theta} : \zeta(s,z,\eta,z',\eta') \leq a \right\}$.  To arrive at this, we take, in their notation, $\epsilon = 1/4$, $\alpha = 2$, $\theta = \max\{\|m'\|_\infty, \|m''\|_\infty\}$, and $a = \bar at$. 
 In addition, what we refer to as $\zeta$, $G$, $\underline R$, and $m$, they refer to as $\rho$, $H$, $K$, and $q$, respectively.  To finish, we need only show that $\Vol(S_{\bar a t}(x,\theta,t))\Vol(S_{\bar a t}(0,\underline\theta,t))$ is uniformly bounded from below.

\begin{lemma}\label{lem:volume}
There exists $\bar a$ and $t_0 > 0$ such that for any $t>t_0$, $x\in\R$ and $\theta \leq \eta_{\gamma_\infty+1}(t)$, $
\Vol(S_{\bar a t}(t,x,\theta))$ and $\Vol(S_{\bar a t}(t, 0,\underline\theta))$
are bounded away from zero.
\end{lemma}

\begin{proof}
Let $B_1(0,\underline \theta) = \{(x,\theta) \in \R\times\Theta: x^2 + (\theta-\underline\theta)^2 < 1\}$.  For any $\overline a \geq 1$ and any $t$ sufficiently large, we have that%
\footnote{To see this fix any $(x,\theta) \in B_1(0,\underline\theta)$ and define $Z(s) = (x\max\{1-s,0\}, \max\{\underline\theta, \theta(1-s) + s\underline\theta\})$.  Notice that $\int_0^t (\dot Z_1^2/4Z_2 + \dot Z_2^2/4 + m(Z_2)) ds \leq x^2/4\underline \theta + (\theta-\underline\theta)^2/4 + \max_{B_1(0,\underline\theta)}m \ll \overline a t$.}
 $B_1(0,\underline \theta) \subset S_{\bar a t} (t,0,\underline \theta)$.  Hence we have that $\Vol(S_{\bar a t}(t,0,\underline \theta)) \geq \Vol(B_1(0,\underline \theta))$. Since $\Vol(B_1(0,\underline \theta)) > 0$, $\Vol(S_{\bar a t}(t,0,\underline\theta))$ is uniformly bounded from below for all $t$ sufficiently large.



We now obtain the bound on $\Vol(S_{\bar a t}(t,x,\theta))$.  Fix $\theta_1$ to be determined.  If $\theta \in [\underline\theta, \theta_1 + 1)$, then we argue as in the previous paragraph to obtain a lower bound on $\Vol(S_{\bar a t}(t,x,\theta))$.  Hence, we may assume that $\theta > \theta_1+1$.  The claim now follows by showing that $(x-t,x+t)\times (\theta_1,\theta_1+1) \subset S_{\bar a t}(t,x,\theta)$, for $\theta_1$ sufficiently large, though independent of $t$.


Indeed, fix $(x',\theta')\in (x-t,x+t)\times(\theta_1,\theta_1+1)$, and let $Z(s) = (Z_1(s), Z_2(s))$ where 
\begin{equation*}
Z_1(s) = x' +\frac{s}{t}(x-x'),
\end{equation*}
and $Z_2$ solves 
\begin{equation*}
Z_2(s) = \theta', \quad \text{ if } s \in [0,s_t]
\qquad \text{and} \qquad
\begin{cases}
\dot Z_2 =  2 \sqrt{m(Z_2(s))},\\
Z_2(t)=\theta, \quad Z_2(s_t) = \theta',
\end{cases}
\text{ if } s\in[s_t,t].
\end{equation*}
Such a trajectory is reasonable since it solves the Euler-Lagrange equations on $[s_t,t]$. However, it is not the optimal trajectory in $\theta$ since we cannot ensure that $s_t = t$, but this is not a problem for our purposes. 
We note that $s_t\geq 0$ exists since
\begin{equation}
\begin{split}
	t-s_t = \int_{s_t}^t \frac{\dot Z_2}{2\sqrt{m(Z_2)}} ds 
		= \int_{\theta'}^\theta \frac{d\eta}{2\sqrt{m(\eta)}}
		\leq\int_{0}^\theta \frac{\sqrt{m(s)} }{2m(\theta_1)} \, ds
		\leq \frac{\Phi(\eta_{\gamma_\infty+1}(t))}{2 m(\theta_1)} \leq \frac{(\gamma_\infty+1)t}{2m(\theta_1)} \leq t,
\end{split}
\end{equation}
where we have used the restriction $\theta \leq \eta_{\gamma_\infty+1}(t)$ and where the last inequality follows by possibly increasing $\theta_1$.
We also used that $Z_2$ is increasing on $[s_t,t]$.

We now estimate $\zeta$ to show that $(x',\theta') \in S_{\overline a t}(t,x,\theta)$.  Indeed, 
\begin{equation}\label{eq:appendix1}
\begin{split}
\zeta&(t,x,\theta,x',\theta')
	\leq \int_0^t \left[\frac{\dot Z_1^2}{4 Z_2} + \frac{\dot Z_2^2}{4} + m(Z_2) \right] ds\\
	&\leq \int_{0}^t \frac{1}{4 \underline\theta} \left( \frac{x'-x}{t} \right)^2 ds + \int_0^{t} \left[ \frac{\dot Z_2^2}{4} + m(Z_2) \right] ds
	\leq \frac{t}{4 \underline\theta} + \int_0^{t} \left[ \frac{\dot Z_2^2}{4} + m(Z_2) \right] ds.
\end{split}
\end{equation}

It remains to estimate the second term in~\eqref{eq:appendix1}:
\begin{multline*}
\int_0^{t} \Big[ \frac{\dot Z_2^2}{4} + m(Z_2) \Big] ds
	= \int_0^{s_t} m(\theta') \, ds + 2\int_{s_t}^{t} m(Z_2) \,ds
	= \int_0^{s_t} m(\theta') \,ds + \int_0^{s_t} \dot Z_2 \sqrt{m(Z_2)} \, ds\\
	= s_t m(\theta') + \int_{\theta'}^{\theta} \sqrt{m(s)} ds
	\leq m(\theta_1 + 1) t + \Phi( \eta_{\gamma_\infty}(t))
	= \left( m(\theta_1 + 1) + \gamma_\infty + 1\right)t.
\end{multline*}
Choosing $\bar a = m(\theta_1+1) + \gamma_\infty + 1 + (4\underline \theta)^{-1}$,
%
we conclude that $(x',\theta') \subset S_{\overline a t}(t,x,\theta)$.  Hence, $(x-t,x+t) \times (\theta_1, \theta_1 +1)\subset S_{\bar a t}(t,x,\theta)$, implying that $\Vol(S_{\bar a t}(t,x,\theta)) \geq O(1)$.
\end{proof}

Plugging this into~\eqref{eq:prelim_heat_kernel_bound} yields the bound in \Cref{lem:Li_Yau}.

\subsubsection*{Discussion of the role of the boundary in \cite[Corollary 3.2]{LiYau}}

While the effect of having a boundary is quite well-understood (see, for example, \cite{Wang}), since the analogue of  \cite[Corollary 3.2]{LiYau} is not explicitly stated in~\cite{Wang}, we briefly outline how to modify the arguments of Li and Yau to obtain it in our setting. In this discussion only, we adopt their notation (the notable changes are that $n$ is denoted $u$, $m$ is denoted $q$, and $\zeta$ is denoted $\rho$).

We begin by noticing that, since we are seeking an upper bound on the propagation in terms of $\int_{\underline\theta}^\theta q(s)ds$, we may assume, without loss of generality, that $q_\theta(\underline\theta) = 0$.  Indeed, if this is not the case, we may replace $q$ with $\tilde q = \chi q$ where $\chi$ is any smooth function satisfying $\chi(\theta) = 1$ if $\theta > \underline \theta + 1$, $\chi(\theta) \in(0,1]$ if $\theta > \underline\theta$, and $\chi(\underline\theta) = 0$.  Since this only {\em increases} the 0th order coefficient $(1-q)$, an upper bound for solutions of the equation with $\tilde q$ in place of $q$ provides an upper bound for solutions of the equation with $q$.  Further, $\tilde q_\theta(\underline\theta) = 0$ and $\int_{\underline\theta}^\theta \tilde q(s) ds = \int_{\underline\theta}^\theta q(s) ds + O(1)$.

The first step in obtaining their Corollary 3.2 in our setting is their Theorem 1.3, which gives a differential inequality for $u$ when there is no boundary.  The proof follows from the work in their Theorem 1.2, which we discuss now.  The key idea is to define a function $F$ in terms of $\log u$ and apply the maximum principle to it.  If the maximum of $F$ is in the interior of the manifold, their argument applies directly;  that is, after choosing a parameter $\alpha>0$ carefully, the upper bound comes directly from the fact that, at the maximum, $F_t - \Delta F \leq 0$, along with some computations.  On the other hand, we rule out the maximum occuring on the boundary with the following proof by contradiction.  The Hopf maximum principle implies that, if the maximum were on the boundary, $\partial_n F > 0$, where $n$ is the outward normal vector.  Then, following the computation in their Theorem 1.1, we obtain $\partial_n F = -2 II(\nabla \log u, \nabla \log u) - \alpha t\partial_n q$, where $II$ is the second fundamental form.  This is the equation below (1.6) in~\cite{LiYau}, where the $\partial_n q$ term is an additional term arising in our setting.   In view of paragraph above, $\partial_n q=0$.   Also, as in their setting, we observe that $II\geq 0$.  To see this, we point out that our domain is geodesically convex, i.e.~all geodesics remain within the domain.  This is a consequence of~\cite[Lemma A.2.(iii)]{HendersonPerthameSouganidis}, which shows that the geodesics, denoted $\gamma$, have positive second coordinate; heuristically, it is true since the metric $g$ rewards paths with large $\theta$.  A consequence of the geodesic convexity is that $II \geq 0$.  We conclude that $\partial_n F \leq 0$, which is a contradiction.  Hence, the conclusion of Theorem 1.3 holds in our setting.

The second step is in obtaining Theorem 2.2, a Harnack inequality, from Theorem 1.3.  Since the proof does not ``see'' the boundary and uses only Theorem 1.3, which we have outlined how to obtain in our setting, it follows that Theorem 2.2 holds in our setting as well.

The third step is to obtain Lemma 3.1, an equation for the action $\zeta$.  This is standard in the Hamilton-Jacobi and physics literature.  Since $q_\theta = 0$ and since the metric $g$ rewards paths with larger $\theta$, it is easy to check that any minimizing path of the action, $\zeta$, does not touch the boundary $\R\times\{\underline\theta\}$.  Hence, the standard arguments apply and the identity in Lemma 3.1 holds.

From here, they use Lemma 3.1 to obtain Lemma 3.2.  Then they apply all the above-mentioned results to obtain Theorem 3.1, from which the result that we use, Corollary 3.2, follows.  In each of these steps, the boundary plays no role.  Hence, the conclusion of Corollary 3.2 holds in our setting.

\bibliographystyle{abbrv}
\bibliography{refs-tradeoff}

\end{document}